\documentclass[11pt, twoside]{amsart}

\title[Twists of graded algebras in monoidal categories]{Twists of graded algebras in monoidal categories}

\author{Fernando Liu Lopez}
\address{Department of Mathematics, Rice University,
P.O. Box 1892, Houston, TX 77005-1892, USA}
\email{fcl2@rice.edu}

\author{Chelsea Walton}
\address{Department of Mathematics, Rice University,
P.O. Box 1892, Houston, TX 77005-1892, USA}
\email{notlaw@rice.edu}

\usepackage{amsmath}
\usepackage{amsfonts}
\usepackage{amsthm}
\usepackage{amssymb,bbm}
\usepackage{dsfont}
\usepackage{stmaryrd}
\usepackage[english]{babel}
\usepackage{url}
\usepackage{fancyhdr}
\usepackage{graphicx}
\usepackage{verbatim}
\usepackage[normalem]{ulem}
\usepackage[shortlabels]{enumitem}
\newcommand{\stkout}[1]{\ifmmode\text{\sout{\ensuremath{#1}}}\else\sout{#1}\fi}
\usepackage[
colorlinks=true,
linkcolor=black, 
anchorcolor=black,
citecolor=black,
urlcolor=black, 
]{hyperref}
\usepackage[all,cmtip]{xy}
\usepackage{geometry}
\usepackage[dvipsnames]{xcolor}

\usepackage[utf8]{inputenc}
\usepackage{tikz}
\usepackage{tikz-cd}
\usepackage{quiver}

\allowdisplaybreaks

\definecolor{forest}{rgb}{0.0, 0.5, 0.0}

\usepackage{enumitem}

\usepackage[notcite,notref,final]{showkeys}

\usepackage{calrsfs}
\DeclareMathAlphabet{\cal}{OMS}{zplm}{m}{n}

\DeclareMathAlphabet{\mathsf}{OT1}{cmss}{m}{n} % to math mathsf thinner

\newcommand{\Id}{\textnormal{Id}}
\newcommand{\id}{\textnormal{id}}
\newcommand{\one}{\mathbbm{1}}

\newcommand{\Hom}{\textnormal{Hom}}

\newcommand{\act}{\triangleright}

\newcommand{\Alg}{\mathsf{Alg}}

\newcommand{\cA}{\cal{A}}
\newcommand{\cC}{\cal{C}}
\newcommand{\cD}{\cal{D}}

\newcommand{\cE}{\cal{E}}

\newcommand{\cV}{\cal{V}}
\newcommand{\cW}{\cal{W}}

\newcommand{\textand}{\phantom{mm}\textnormal{and}\phantom{mm}}

\numberwithin{equation}{section}

\newtheorem{theorem}{Theorem}[section]
\newtheorem{claim}[theorem]{Claim}
\newtheorem{proposition}[theorem]{Proposition}
\newtheorem{definitionproposition}[theorem]{Definition-Proposition}
\newtheorem{corollary}[theorem]{Corollary}
\newtheorem{lemma}[theorem]{Lemma}

\newtheorem{theorem*}{Theorem}

\theoremstyle{definition}
\newtheorem{definition}[theorem]{Definition}
\newtheorem{notation}[theorem]{Notation}

\newtheorem{example}[theorem]{Example}
\newtheorem{remark}[theorem]{Remark}

\newtheorem{hypothesis}[theorem]{Hypothesis}
\newtheorem{convention}[theorem]{Convention}

\makeatletter              % This sequence of commands will
\let\c@equation\c@theorem  % incorporate equation numbering
\makeatother
\numberwithin{equation}{section}

\newcommand{\skp}{\vspace{0.07in}}

\makeatletter
\@namedef{subjclassname@2020}{%
  \textup{2020} Mathematics Subject Classification}
\makeatother


\begin{document}

\maketitle

%%%%%%%%%%%%%%%%%%%%%%%
% TEMPORAL TOOLS FOR WRITING
% Comment in the final version
%%%%%%%%%%%%%%%%%%%%%%%
%\begin{center}
%\fll{[FLL's edits {\tt $\backslash$fll}]}, \quad
%\cw{[CW's edits {\tt $\backslash$cw}]},\\
%\blue{items to be filled in {\tt $\backslash$blue} %(when it needs to be highlighted later)},  \quad
%\red{math concerns in {\tt $\backslash$red}},\\
%\quad [comments in brackets], \quad
 %strike-out \stkout{text} with {\tt %$\backslash$stkout}.
%\end{center}

\begin{abstract}
Zhang twists are a common tool for deforming graded algebras over a field in a way that preserves important ring-theoretic properties. We generalize Zhang twists to the setting of closed monoidal categories equipped with their self-enriched structure. Along the way, we prove several key results about algebraic structures in closed monoidal categories missing from the literature. We use these to ultimately prove Morita-type results, showcasing when graded algebras with equivalent categories of graded modules can be related by Zhang twists.
\end{abstract}

%\setcounter{tocdepth}{3}

%\begin{changemargin}{0.5cm}{0.5cm} 
%\tableofcontents
%\end{changemargin}

%%%%%%%%%%%%%%%%%%%%%%%%%%
%%%%%%%%%%%%%%%%%%%%%%%%%%
%%%%%%%%%%%%%%%%%%%%%%%%%%

\section{Introduction}\label{sec:intro}

Twists of graded algebras over a field were introduced by Zhang \cite{Zhang1996} to study certain categories of graded modules over graded algebras in Noncommutative Projective Algebraic Geometry (NCPAG), and since then the framework of twisted algebras has proven useful in its own right. It was shown  that many important properties of graded algebras are preserved under the procedure that is now known as {\it Zhang twist}, including several growth measures (e.g., Gelfand-Kirillov dimension), ring-theoretic notions (e.g., Noetherian, domain), and homological conditions and measures (e.g., global dimension, Gorenstein conditions) \cite[Theorem~1.3]{Zhang1996}. On the other hand, properties of algebras can also be preserved via equivalent categories of modules (e.g., via Morita equivalence). The main result of \cite{Zhang1996} is that twisting algebras is related to having a type of Morita equivalence; this is useful as the former is more fruitful computationally than the latter. Zhang's results have since been revamped in the context of  Morita theory \cite{Sierra2011}, have been used in the classification of Artin-Schelter regular algebras (the central algebras of NCPAG) and to understand noncommutative projective geometry itself \cite{Zhang1998, Cassidy1999, Stephenson2000, Rogalski2004, NVZ2011, VdB2011, RogalskiZhang2012, ZhouLu2014, Pym2015, LecoutreSierra2019, MoriUeyama2021}, to understand universal quantum groups \cite{BNY2018, HNUVVW2022pp,HNUVVW2022}, and more \cite{SierraWalton2016, FRS2020}.
The goal of our work is to generalize Zhang's twisting method for graded algebras, and to study the corresponding categories of graded modules, in the monoidal setting. 
 
\skp
  
 Recall that a {\it monoidal category} is a category $\cC$, equipped with a bifunctor $\otimes: \cC \times \cC \to \cC$ and an object $\one \in \cC$ such that the triple $(\cC, \otimes, \one)$ mimics the structure of a monoid. Examples include the monoidal category of $\Bbbk$-vector spaces $(\mathsf{Vec}_\Bbbk, \otimes_\Bbbk, \Bbbk)$, for a field $\Bbbk$, and the monoidal category of endofunctors $(\mathsf{End}(\cA), \circ, \Id_\cA)$, for a category $\cA$.
 Moreover, an {\it algebra} in a monoidal category $\cC$ is an object $A$ in $\cC$, equipped with morphisms $m: A \otimes A \to A$ and $u: \one \to A$ in $\cC$, such that the triple $(A,m,u)$ also mimics the structure of a monoid. Examples include $\Bbbk$-algebras as algebras in $\mathsf{Vec}_\Bbbk$, and {\it monads} over $\cA$ as algebras in $\mathsf{End}(\cA)$. These can be turned into $G$-graded structures, for $G$ a group, by swapping  $(\cC, \otimes, \one)$ for a certain monoidal category of $G$-tuples of objects in~$\cC$. Further, one can form graded modules over graded algebras in monoidal categories, where morphisms are $G$-tuples of morphisms in $\cC$; that is, the grading of morphisms here are preserved. We denote such categories of graded algebras in $\cC$ and graded right $A$-modules in $\cC$, respectively, by $\mathsf{GrAlg}(\cC)$ and $\mathsf{GrMod}(\cC)_A$. Note that mild conditions on $\cC$ given Hypothesis~\ref{hypotheses} are assumed throughout.
 
\skp
   
One way that we compare graded algebras in monoidal categories is via twists \`{a} la Zhang \cite{Zhang1996}. A {\it twisting system} $\tau$ on a graded algebra $A = (A_g)_{g \in G}$ in $\cC$  is a collection of isomorphisms in~$\cC$, namely $(\tau_d(g): A_g \overset{\sim}{\longrightarrow} A_g)_{d,g \in G}$, satisfying certain compatibility  conditions [Definition~\ref{def:twist}]. Then with a twisting system, one can {\it twist} $A$ to form another graded algebra $A^\tau$ in $\cC$ with altered multiplication and unit morphisms [Definition-Proposition~\ref{defprop:Atau}]. Now  consider the following terminology. 
 
 \begin{definition}[Definition~\ref{def:twistequiv}] \label{def:twistequiv-intro}
 We say that graded algebras $A$ and $B$ in $\cC$ are {\it twist equivalent} if there exists a twisting system $\tau$ on $A$ such that $A^\tau \cong B$ as graded algebras in $\cC$.   
\end{definition}

Next, we study categories of graded modules in monoidal categories. To do so, we consider enrichments. A category $\cC$ is said to be {\it enriched} over a monoidal category $(\cV, \otimes^{\cV}, \one^{\cV})$ if its collections of morphisms between two objects is an object of $\cV$. Compositions of morphisms are then given via $\otimes^\cV$ (see Section~\ref{sec:enrich}). For instance, when $\cC$ is the monoidal category $\mathsf{Vec}_\Bbbk$,  both $\mathsf{GrAlg}(\mathsf{Vec}_\Bbbk)$ and $\mathsf{GrMod}(\mathsf{Vec}_\Bbbk)_A$ are $\mathsf{Vec}_\Bbbk$-enriched (which is the same as being $\Bbbk$-linear).  Our next contribution is a result about enriched categories of graded modules.

\begin{proposition}[Corollary~\ref{cor:GrModCA-main}] \label{prop:grmodA-intro}
Let $(\cC,\otimes,\one)$ be a right-closed monoidal category with equalizers, and take a graded algebra $A$ in $\cC$. Then, the category  $\mathsf{GrMod}(\cC)_A$ is $\cC$-enriched.  \qed\end{proposition}

Here, a monoidal category $(\cC, \otimes, \one)$ is {\it right-closed} if the endofunctor $- \otimes Y: \cC \to \cC$ admits a right adjoint, for each $Y \in \cC$. For instance, $\mathsf{Vec}_\Bbbk$ is right-closed via standard Tensor-Hom adjunction. Proposition~\ref{prop:grmodA-intro} is achieved by first showing that in the ungraded case the category of right \linebreak $A$-modules in $\cC$ is $\cC$-enriched,  when $\cC$ is right-closed [Theorem~\ref{thm:closed}]. Then we move to the graded setting by showing that the monoidal category $\mathsf{Gr}(\cC)$ of graded objects in $\cC$ is right-closed and has equalizers when $\cC$ satisfies these conditions [Theorem~\ref{thm:closed2}]; this implies that $\mathsf{GrMod}(\cC)_A$ is $\mathsf{Gr}(\cC)$-enriched. Finally, we transfer enrichments from over $\mathsf{Gr}(\cC)$ to over $\cC$ via a monoidal functor $\mathsf{Gr}(\cC) \to \cC$ [Proposition~\ref{prop:e-functor-monoidal}, Lemma~\ref{lem:tranport-enrichment}].

\skp

Now in the enriched setting, we can discuss when two such categories are the same. Namely, $\cV$-categories $\cC$ and $\cD$ are said to be {\it $\cV$-equivalent} when there is a fully-faithful and essentially surjective functor between them in the enriched sense [Definition~\ref{def:V-equiv}]. Here, we write  $\cC \simeq^{\cV} \cD$. This brings us to the second way that we will compare graded algebras in monoidal categories. 

\begin{definition} \label{def:ZM-equiv-intro}
Two graded algebras $A$ and $B$ in $\cC$
are said to be {\it Zhang-Morita equivalent} (or, {\it ZM-equivalent}) 
if $\mathsf{GrMod}(\cC)_A \simeq^{\cC} \mathsf{GrMod}(\cC)_B$.
\end{definition}

This captures the module-theoretic notion of sameness for the graded algebras studied in Zhang \cite{Zhang1996}, which is different than ordinary Morita equivalence  [Remark~\ref{rem:Morita}]. Indeed when  $\cV = \cC = \mathsf{Vec}_\Bbbk$, an equivalence between enriched categories is the same as an equivalence that preserves $\Bbbk$-linearity; the latter is imposed  in \cite{Zhang1996}. Now we generalize  \cite[Theorems~3.1 and~3.4]{Zhang1996} as follows. 
 
\begin{theorem}[Theorem~\ref{thm:mainfwd}, Corollary~\ref{cor:backward}] \label{thm:main-intro}
Take graded algebras $A$ and $B$ in $\cC$.
\begin{enumerate}[\upshape (a)]
\item If $A$ and $B$ are twist equivalent, then they are ZM-equivalent.

\smallskip

\item If $A$ and $B$ are ZM-equivalent such that a certain $G$-grading condition is satisfied, then they are twist equivalent. \qed
\end{enumerate} 
\end{theorem}

Part (a) is the easier direction; in fact, an isomorphism of categories,  $\mathsf{GrMod}(\cC)_A \cong \mathsf{GrMod}(\cC)_B$, is established here.  The proof of part (b) follows from first analyzing the {\it $G$-grading} on the categories of the form $\mathsf{GrMod}(\cC)_A$ [Section~\ref{sec:graded-shifts}]. Then, we introduce some graded algebras $\Gamma^{\cC}(A)$ in $\cC$, for the regular graded module $A \in \mathsf{GrMod}(\cC)_A$, that are analogous to the graded $\Bbbk$-algebras $\Gamma(A)$ in \cite[Section~3]{Zhang1996}. The $G$-grading hypothesis implies that when $A$ and $B$ are ZM-equivalent, we get that $\Gamma^{\cC}(A)$ and $\Gamma^{\cC}(B)$  are twist equivalent [Theorem~\ref{thm:Gamma-twist}]. Lastly, we show that $\Gamma^{\cC}(A)$ is realized as an endomorphism algebra in $\cC$ [Example~\ref{ex:GammaCA}], which is in turn, is isomorphic to $A$ as graded algebras in $\cC$ [Theorem~\ref{thm:endomalg}, Corollary~\ref{cor:Gamma-A}].

\medskip
 
 \noindent {\bf Further directions.} We discuss here directions for future research that may be of interest. 
\smallskip

\noindent {\bf I.} \; As mentioned in the beginning of this section, many algebraic properties of graded algebras over a field are preserved under twist.  It is natural to then investigate, or even develop, these properties in the monoidal setting.  Growth measures should always be twist-invariant, for instance.

\skp

\noindent {\bf II.}\; Graded monads are used in theoretical computer science, specifically in {\it trace semantics} \cite{MPS15}, in the {\it semantics of  effect systems} \cite{Katsumata14, MOP16, BMU2022},  in the {\it semantics of concurrent systems}  \cite{DMS2019}, and more \cite{FMS21,Forsteretal2023}. 
The results on twisting in this article will apply when the monoidal category $\mathsf{End}(\cA)$ satisfies the conditions of Hypothesis~\ref{hypotheses}. So we would obtain a theory of  twisted graded monads and of their  categories of graded modules (e.g., of their graded {\it Eilenberg-Moore categories}) over $\cA$. This is related to the {\it finitary} and {\it algebraic} conditions that are often imposed in these works. Exploring then further connections to theoretical computer science would be an intriguing line of research.  
For instance, {\it depth} in these fields is a growth measure that should be twist-invariant.

\medskip
    
 \noindent {\bf Organization of article.}  In Section~\ref{sec:prelim}, we provide background material and preliminary results on graded algebraic structures in monoidal categories. In Section~\ref{sec:graded-closed-enrich}, we discuss closed monoidal categories and enrichment, especially in the graded setting. In particular, endomorphism algebras in categories of modules in the monoidal setting are examined, gradings on enriched categories are discussed, and Zhang-Morita (ZM-)equivalence is introduced. In Section~\ref{sec:twist},  twisting systems and twisted algebras are introduced for the monoidal setting, and twist equivalence is defined here. The main results of the article are then established in Section~\ref{sec:mainresults}. Many results involve lengthy diagrammatic arguments, and many of  these proofs are deferred to the Appendix.

%%%%%%%%%%%%%%%%%%%%%%%%%%
%%%%%%%%%%%%%%%%%%%%%%%%%%
%%%%%%%%%%%%%%%%%%%%%%%%%%
\section{Graded algebraic structures in monoidal categories}
\label{sec:prelim}

In this part, we provide background material and preliminary results on (graded) algebraic structures in monoidal categories. In Section~\ref{sec:algmoncat}, we review monoidal categories and algebras and modules within them. Then we introduce graded algebraic structures in monoidal categories in Section~\ref{sec:gralg-monoidal}.
%, and the doubly graded versions in Section~\ref{sec:doubgralg-monoidal}.

\subsection{Algebraic structures in monoidal categories} 
\label{sec:algmoncat}
Consider the categorical structures below. We refer the reader to \cite[Section~3.1]{HeckSch2020} for further details.

%https://www.dropbox.com/s/42j6hqvu80pft3f/Heckenberger-Schneider-RootSystems.pdf?dl=0

\begin{definition}
A {\it monoidal category} is a category $\cC$ equipped with a bifunctor $\otimes\colon  \cC \times \cC \to \cC$, a natural isomorphism $a_{X,Y,Z}\colon  (X \otimes Y) \otimes Z \overset{\sim}{\to} X \otimes (Y \otimes Z)$ for each $X,Y,Z \in \cC$, an object $\one \in \cC$, and natural isomorphisms $l_X\colon  \one \otimes X \overset{\sim}{\to} X$ and $r_X\colon  X  \otimes \one \overset{\sim}{\to} X$ for each $X \in \cC$, such that the pentagon and triangle axioms hold. 
\end{definition}

\begin{hypothesis}[$\cC, \otimes, \one$] \label{hyp:strict}
Unless stated otherwise, by MacLane's strictness theorem, we will assume that all monoidal categories are {\it strict} in the sense that $(X \otimes Y) \otimes Z = X \otimes (Y \otimes Z)$ and $\one \otimes X  = X = X \otimes \one$, for all $X, Y, Z \in \cC$; that is, $a_{X,Y,Z},\; l_X,\; r_X$ are identity maps. As a result, monoidal categories will be denoted by $(\cC, \otimes, \one)$. 
\end{hypothesis}

\begin{remark}
    If $\cC$ is monoidal, it follows from the functoriality of $\otimes$ that, for any morphisms
    $f:X\to X'$ and $g:Y\to Y'$, the following equalities holds:
        \[ f\otimes g = (f\otimes\id_Y)(\id_X\otimes g) = (\id_X\otimes g)(f\otimes\id_Y).\]
    We refer to the equalities above as {\it level exchange} and use them throughout the paper.
\end{remark}

Next, we discuss how to move from one monoidal category to another.

\begin{definition}[$F, F^{(2)}, F^{(0)}$]
Let $\cC:=(\cC, \otimes^{\cC}, \one^{\cC})$ and  $\cD:=(\cD,\otimes^{\cD},\one^{\cD})$ be monoidal categories.
 A \textit{monoidal functor} from $\cC$ to $\cD$ is a functor $F: \cC \to \cD$, equipped with a natural transformation 
 \[F^{(2)}= \{F^{(2)}_{X,X'} : F(X)\otimes^{\cD} F(X') \to F(X \otimes^{\cC} X')\}_{X,X' \in \cC}\] 
 in $\cD$,  and a morphism
 $F^{(0)}: \one^{\cD} \to F(\one^{\cC})$ in $\cD$, that satisfy associativity and unitality constraints.
\end{definition}

Now we introduce algebraic structures in monoidal categories. 

\begin{definition}[$(A,m,u), \; \mathsf{Alg}(\cC)$]\label{def:alg} 
Let $(\cC, \otimes, \one)$ be a monoidal category.
\begin{enumerate}[(a)]
\item An \textit{algebra} in $(\cC, \otimes, \one)$ is a triple $(A,m,u)$ consisting of an object  $A\in \cC$, along with morphisms $m: A \otimes  A \rightarrow A$, $u: \one \rightarrow A$  in $\cC$, satisfying associativity and unitality constraints: $m(m\otimes \text{id}_A) = m(\text{id}_A \otimes m)$, and $m(u\otimes \text{id}_A) =\text{id}_A$ = $m(\text{id}_A\otimes u)$. 

\skp

\item A \textit{morphism} of algebras $\varphi:(A,m^A,u^A)\rightarrow(B,m^B,u^B)$ is a morphism $\varphi:A\rightarrow B$ in $\cC$, such that $\varphi \hspace{0.02in} m^A = m^B (\varphi \otimes \varphi)$ and $\varphi \hspace{0.02in} u^A=u^B$. We call $\varphi$ an {\it isomorphism} of algebras if, further, $\varphi$ is an isomorphism in $\cC$.

\skp

\item Algebras in $\cC$ and their morphisms form a category, which we denote by $\textsf{Alg}(\cC)$.
\end{enumerate}
\end{definition}

Monoidal functors transport algebras in monoidal categories. Namely, we have the result below.

\begin{proposition} \label{prop:algtoalg}
If  $(F,F^{(2)},F^{(0)}):\cC \to \cD$ is a monoidal functor and if $(A,m,u) \in {\sf Alg}(\cC)$, then $\left(F(A),  \; F(m)F^{(2)}_{A,A},\; F(u)F^{(0)}\right) \in {\sf Alg}(\cD)$. \qed
\end{proposition}

We can also define modules over algebras in monoidal categories. 

\begin{definition}[$(M, \rho),  \mathsf{Mod}(\cC)_A$]\label{def:module} 
Take $A:=(A,m,u)$ an algebra in a monoidal category $\cC$. 
\begin{enumerate}[(a)]
\item A \textit{right $A$-module} in $\cC$ is a pair $(M, \rho)$ consisting of  an object $M \in \cC$ along with a morphism $\rho:M \otimes A \rightarrow M$ in $\cC$ satisfying associativity and unitality constraints: 
    \[ \rho(\text{id}_M \otimes m) = \rho(\rho \otimes \text{id}_A) \textand \rho(\text{id}_M \otimes u)=\text{id}_M.\]
 
 \item A \textit{morphism} of right $A$-modules $\varphi:(M,\rho^M) \rightarrow (N,\rho^N)$ is a morphism $\varphi:M \rightarrow N$ in $\cC$ such that $\rho^N (\varphi \otimes \text{id}_A) = \varphi \hspace{0.02in} \rho^M$. 
 
 \skp
 
\item Right $A$-modules in $\cC$ and their morphisms form a category, which we denote by $\mathsf{Mod}(\cC)_A$.
\end{enumerate}
\end{definition}

%%%%%%%%%%%%%%%%%%%%%%%%%%

\subsection{Graded algebraic structures in monoidal categories} \label{sec:gralg-monoidal} Now we introduce graded versions of the algebraic structures in the previous section.

\begin{notation} \label{not:G-C} 
 Take $G$ to denote a group, with neutral element $e$, throughout.
\end{notation}

\begin{definition}[$(A,\{m_{g,h}\}_{g,h \in G},u_e), \mathsf{GrAlg}(\cC)$]\label{def:GrAlg(C)} 
Let $(\cC, \otimes, \one)$ be a monoidal category.
\begin{enumerate}[(a)]
    \item A {\it $G$-graded algebra in $\cC$} is a triple $(A,\{m_{g,h}\}_{g,h \in G},u_e)$ consisting of a tuple of objects $A=(A_g)_{g\in G}$ in $\cC$, a collection of  morphisms $\{m_{g,h}:A_g \otimes A_h \to A_{gh}\}_{g,h\in G}$ in $\cC$, and a  morphism $u_e:\one\rightarrow A_e$ in $\cC$, satisfying the following associativity and unitality constraints:

  \skp
  
    \begin{itemize}
        \item $m_{gh,k}(m_{g,h} \otimes \id_{A_k}) = m_{g,hk}(\id_{A_g} \otimes m_{h,k})$, for all $g,h,k \in G$, and 
        \skp
        \item $m_{e,g}(u_e \otimes \id_{A_g})  = \id_{A_g} = m_{g,e}(\id_{A_g} \otimes u_e)$, for all $g \in G$.
    \end{itemize}
    
    \skp
    
    \item A {\it morphism} of $G$-graded algebras $\varphi:(A,m^A_{g,h},u^A_e)\rightarrow(B,m^B_{g,h},u^B_e)$ consists of a tuple of morphisms $\varphi=(\varphi_g: A_g \to B_g)_{g\in G}$ in $\cC$, such that $m^B_{g,h}\hspace{0.02in}(\varphi_g \otimes \varphi_h) = \varphi_{gh}\hspace{0.02in} m^A_{g,h}$, for all $g,h \in G$, and $\varphi_e \hspace{0.02in} u^A_e = u^B_e$.

       \skp

    \item $G$-graded algebras in $\cC$ and their morphisms form a category,  denoted by $\mathsf{GrAlg}(\cC)$.
\end{enumerate}
\end{definition}

\begin{definition}[$(M,\{\rho_{g,h}\}_{g,h\in G}), \; \mathsf{GrMod}(\cC)_A$]\label{def:gr-mod} 
Let $(\cC, \otimes, \one)$ be a monoidal category.

\begin{enumerate}[(a)]
 \item A {\it $G$-graded right $A$-module in $\cC$} is a pair $(M,\{\rho_{g,h}\}_{g,h \in G})$ consisting of a tuple $M=(M_g)_{g\in G}$ of objects in $\cC$ and a collection of morphisms $\{\rho_{g,h}:M_g \otimes A_h \to M_{gh}\}_{g,h\in G}$ in $\cC$ satisfying the associativity and unitality constraints: 

 \skp
 
    \begin{itemize}
        \item $\rho_{gh,k}(\rho_{g,h} \otimes \id_{A_k}) = \rho_{g,hk}(\id_{M_g} \otimes m_{h,k})$, for all $g,h,k \in G$, and 
    
        \skp
    
        \item $\rho_{g,e}(\id_{M_g} \otimes u_e) = \id_{M_g}$, for all $g \in G$.
    \end{itemize}
    
 \skp 

 \item A  {\it morphism} of graded $A$-modules $\varphi:(M, \; \{\rho^M_{g,h}\}) \to (N,\; \{\rho^N_{g,h}\})$ is a tuple of morphisms $\varphi=(\varphi_g: M_g \to N_g)_{g\in G}$ in $\cC$, such that $\rho^N_{g,h}  \hspace{0.02in} (\varphi_g \otimes \id_{A_h}) = \varphi_{gh} \hspace{0.02in} \rho^M_{g,h}$, for all $g,h \in G$.
    
 \skp 

\item $G$-graded right $A$-modules and their morphisms form  a category, denoted by $\mathsf{GrMod}(\cC)_A.$ 
\end{enumerate} 
\end{definition}

Next, we provide an alternative notion of graded algebraic structures via the category below.

\begin{definition}[$\mathsf{Gr}(\cC)$]\label{def:barC} 
For any category $\cC$, we denote by $\mathsf{Gr}(\cC):=\prod_{g\in G}\cC=\cC \times \cdots \times \cC$ the product category described below.
 \skp
 
\begin{enumerate}[(a)]
    \item Objects are tuples of objects in $\cC$, denoted as $X:= (X_g)_{g\in G}$.

    \skp

    \item Morphisms are tuples of morphisms in $\cC$: namely, $\varphi: X \to Y$ is given by $(\varphi_g: X_g \to Y_g)_{g \in G}$. 
\end{enumerate}

\skp

\noindent We refer to $\mathsf{Gr}(\cC)$ as the {\it category of $G$-graded objects in $\cC$}.
\end{definition}

To give $\mathsf{Gr}(\cC)$ a monoidal structure, we need the next hypothesis. 

\begin{hypothesis}\label{hypotheses} 
From now on, assume the monoidal category $(\cC, \otimes, \one)$ satisfies the following.

\begin{enumerate}[(a)]
\item Assume $\cC$ is additive. If the group $G$ is infinite,  assume that $\cC$ has $G$-indexed biproducts. Biproducts will be denoted by $\oplus$, the zero object by $0$, and  zero morphisms by $\vec{0}$.

\skp

\item Assume that $\otimes$ is additive on morphisms. 
In particular, it follows that $f \otimes \vec{0} = \vec{0} = \vec{0} \otimes f$, for all morphisms $f \in \cC$.
\end{enumerate}
\end{hypothesis}

The next result  then is straight-forward to verify.

\begin{proposition} \label{prop:gr(C)-monoidal} 
The category $\mathsf{Gr}(\cC)$ admits a monoidal structure $(\bar{\otimes},\bar{\one})$ as follows.

\begin{enumerate}[\upshape(a)] 
\item The monoidal product  $X \; \bar{\otimes} \; Y$ is given by $(X \; \bar{\otimes} \; Y)_g:= \bigoplus_{p \in G} (X_p \otimes Y_{p^{-1} g})$, for each $g \in G$.

\skp

\item The monoidal unit  $\bar{\one}$ is given by $\bar{\one}_e :=\one$, and $\bar{\one}_{g} = 0$, for each $g \in G$ not equal to $e$. \qed
\end{enumerate}
\end{proposition}

The monoidal structure on $\mathsf{Gr}(\cC)$ above is referred to as the {\it Cauchy monoidal structure on graded objects}; cf. \cite[Definition~2.1]{AguMah2010}. The next result is also straight-forward.

\begin{proposition}[$\langle -\rangle_e$] \label{prop:e-functor-monoidal} The functor $\langle -\rangle_e:\mathsf{Gr}(\cC)\rightarrow\cC$ projecting graded objects (and graded morphisms) onto the neutral degree, via $X=(X_g)_{g\in G}\mapsto X_e$, is monoidal. Here,
\[ \textstyle
\langle -\rangle_e^{(2)}\phantom{}_{X,Y}:X_e\otimes Y_e\longrightarrow (X\bar{\otimes} Y)_e =\bigoplus_p X_p \otimes Y_{p^{-1}}\] is the canonical coproduct inclusion, and $\langle -\rangle_e^{(0)}:\one_\cC\rightarrow (\bar{\one})_e=\one_\cC$ is the identity morphism. \qed
\end{proposition}

% [use $p,q,r,s$ always for coproduct indices in the monoidal product, and reserve $g,h,k,\ell$ for grading degrees?] 

Next, we see that algebras and modules in the monoidal category $\mathsf{Gr}(\cC)$ coincide with the constructions in Definitions~\ref{def:GrAlg(C)} and \ref{def:gr-mod}.

\begin{proposition} \label{prop:alg-bar} 
We have an isomorphism of categories: $\mathsf{Alg}(\mathsf{Gr}(\cC)) \; \cong \; \mathsf{GrAlg}(\cC)$. \qed
\end{proposition}

\begin{proposition}\label{prop:mod-bar} 
For every $(A,m,u)\in\mathsf{Alg}(\mathsf{Gr}(\cC))$ there exists $(A,\{m_{g,h}\},u_e)\in\mathsf{GrAlg}(\cC)$ inducing an isomorphism of categories: $\mathsf{Mod}(\mathsf{Gr}(\cC))_A \;  \cong \; \mathsf{GrMod}(\cC)_A$.  \qed
\end{proposition}

The proof of Proposition~\ref{prop:alg-bar} is provided in Appendix~\ref{sec:prop:alg-bar}. Proposition~\ref{prop:mod-bar} is proven in the same manner, with $A$ replaced by $M$.

\begin{convention} \label{conv:algs} 
Using Propositions~\ref{prop:alg-bar} and~\ref{prop:mod-bar} above,  we freely convert between algebras in~$\mathsf{Gr}(\cC)$ and graded algebras in $\cC$ from now on. We also proceed similarly for modules.  This will be done often without mention, as depicted below.
\begin{align*}
\mathsf{Alg}(\mathsf{Gr}(\cC)) \ni \; (A, m, u) \; &\leftrightsquigarrow \;  (A,\; \{m_{g,h}\}_{g,h \in G}, \; u_e) \; \in \mathsf{GrAlg}(\cC).\\    
\hspace{.3in} \mathsf{Mod}(\mathsf{Gr}(\cC))_A \ni \; (M, \rho) \; &\leftrightsquigarrow \;  (M,\; \{\rho_{g,h}\}_{g,h\in G}) \; \in \mathsf{GrMod}(\cC)_A.
\end{align*}
\end{convention}

%%%%%%%%%%%%%%%%%%%%%%%%%%
%%%%%%%%%%%%%%%%%%%%%%%%%%
%%%%%%%%%%%%%%%%%%%%%%%%%%

\section{Graded closed monoidal categories and enrichment}
\label{sec:graded-closed-enrich}

Here, we provide background information and preliminary results on closed monoidal categories and enrichment, especially in the graded setting. We recall enriched categories in Section~\ref{sec:enrich}, and closed monoidal categories in Section~\ref{sec:prelim-closed}. Then, we provide results on enriched categories of modules in the monoidal setting in Section~\ref{sec:enriched-mod}. Next, endomorphism algebras in categories of modules in the monoidal setting are studied in Section~\ref{sec:endomalg}. Then, we introduce closure in the graded monoidal setting in Section~\ref{sec:graded-closed}, and discuss graded monoidal categories with shifts in Section~\ref{sec:graded-shifts}. Finally, Zhang-Morita equivalence is introduced in Section~\ref{sec:ZM-equiv}.

\subsection{Enriched categories} \label{sec:enrich}
The material here is borrowed from \cite[Sections~1.2, 1.3, and~1.11]{Kelly2005}. Let $(\cV,\otimes^\cV,\one^\cV)$ be a (strict) monoidal category. Consider the following notion.

\begin{definition} 
\label{def:enr-cat}
A {\it $\cV$-category } $\cC$ (or a {\it category $\cC$ enriched over $\cV$}) consists of the following data:
\begin{enumerate}[(a)]
\item A class of objects of $X,Y,\dots\in\cC$,

\skp
        
\item A {\it Hom-object} $\cC(X,Y):=\cC^\cV(X,Y)$ in $\cV$, for each pair of objects $X,Y\in\cC$,

\skp

\item An identity morphism $\kappa^\cV_X:\one^\cV\rightarrow\cC(X,X)$ in $\cV$, for each object $X\in\cC$,
        
\skp
        
\item A composition morphism $\gamma^\cV_{X,Y,Z}: \cC(Y,Z)\otimes^\cV \cC(X,Y) \rightarrow \cC(X,Z)$ in $\cV$, for each triple of objects $X,Y,Z \in\cC$,
\skp
\end{enumerate}
satisfying the following associativity and unitality constraints:
\begin{itemize}
\skp
\item $\gamma^\cV_{W,X,Z}(\gamma^\cV_{X,Y,Z}\otimes \id_{\cC(W,X)}) \; = \;
\gamma^\cV_{W,Y,Z}(\id_{\cC(Y,Z)}\otimes\gamma^\cV_{W,X,Y})$ for all $W,X,Y,Z\in\cC$,
                 
\skp
        
\item $\gamma^\cV_{X,Y,Y}(\kappa^\cV_Y\otimes\id_{\cC(X,Y)}) \; = \;
\id_{\cC(X,Y)}$  for all $X,Y\in\cC$, and
         
\skp

\item $\gamma^\cV_{X,X,Y}(\id_{\cC(X,Y)}\otimes\kappa^\cV_X) \;  = \;
\id_{\cC(X,Y)}$  for all $X,Y\in\cC$.
\end{itemize}
\end{definition}

\smallskip

Note that  a $\cV$-category $\cC$ is not necessarily an ordinary category, i.e., we do not necessarily have a collection of morphisms $\Hom_\cC(X,Y)$ between objects $X,Y \in \cC$. But consider the notion below.
%We initially only have {\it Hom-objects} $\cC(X,Y) \in \cV$. 

\begin{definition}[$\cC_0$] 
%\cite[Section~1.3]{Kelly2005}
Let $\cC$ be a $\cV$-category. The {\it underlying category}  of $\cC$ is defined as the category $\cC_0$ defined by the following data. 

\begin{enumerate}[(a)]
\item The same class of objects  $X,Y, \dots$ as $\cC$.

\skp

\item A morphism $f:X\rightarrow Y$ in $\cC_0$, for each morphism $\one^\cV\rightarrow \cC(X,Y)$ in $\cV$.

\skp
        
\item The identity morphism of $X$ corresponds to $\kappa^\cV_X: \one^\cV\rightarrow \cC(X,X)$.

\skp

\item The composition of morphisms $f:X\rightarrow Y$ and $g:Y\rightarrow Z$ corresponds to:

\vspace{-.2in}

\[ 
\xymatrix{
\one^\cV 
\ar[r]^(.4){\cong}
&\one^\cV\otimes\one^\cV
\ar[r]^(.38){g\otimes f}
&\cC(Y,Z)\otimes \cC(X,Y)
\ar[r]^(.62){\gamma_{X,Y,Z}^\cV}
&\cC(X,Z).
}
\]
    \end{enumerate}
\end{definition}

Next, we discuss how to move between $\cV$-enriched categories.

\begin{definition}[$F, \{F_{X,Y}\}$]\label{def:enriched-functor}
Let $\cC$ and $\cD$ be $\cV$-categories. A {\it $\cV$-functor} $F:\cC\rightarrow\cD$ from $\cC$ to $\cD$ consists of:
\begin{enumerate}[(a)]
    \item for each object $X\in\cC$, an object $F(X)\in\cD$,

    \skp
    
    \item for all $X,Y\in\cC$, a morphism $F_{X,Y}:\cC(X,Y)\rightarrow\cD(F(X),F(Y))$ in $\cV$,
    \skp
\end{enumerate}
that respects the identity and composition of the enrichments, i.e.
\begin{itemize}
    \item $F_{X,Z}\circ\gamma_{X,Y,Z}=\gamma_{F(X),F(Y),F(Z)}\circ(F_{Y,Z}\otimes F_{X,Y})$, for all $X,Y,Z\in\cC$, 

    \skp
    
    \item $\kappa_{F(X)}=F_{X,X}\circ\kappa_X$, for all $X\in\cC$.
\end{itemize} 
\end{definition}

With this, we can discuss when two enriched categories are considered to be the same.

\begin{definition}[$\cC \simeq^{\cV} \cD$] \label{def:V-equiv}
Let $\cC$ and $\cD$ be $\cV$-categories. We say that $\cC$ and $\cD$ are {\it $\cV$-equivalent} if there exists a $\cV$-functor $F: \cC \to \cD$ that satisfies the conditions below.
\begin{enumerate}[(a)]
\item $F$ is {\it $\cV$-fully faithful}, that is, $F_{X,Y}$ is an isomorphism in $\cV$, for all $X,Y \in \cC$.
\skp
\item $F$ is {\it $\cV$-essentially surjective}, that is, for each object $Z \in \cD_0$, there exists an object $X \in \cC_0$ such that $F(X) \cong Z$ in $\cD_0$.
\end{enumerate}
In this case, we write $\cC \simeq^{\cV} \cD$.
\end{definition}

The next straight-forward lemma shows how monoidal functors transport enrichments.

\begin{lemma}\label{lem:tranport-enrichment} 
    Let $\cV$ and $\cW$ be monoidal categories, $(F,F^{(2)}, F^{(0)}):\cV\rightarrow\cW$ be a monoidal functor, and suppose that $\cC$ be a $\cV$-category. Then, $\cC$ admits the structure of a  $\cW$-category with:
    \begin{enumerate}[\upshape (a)]
        \item The same objects $X,Y,\dots\in\cC$;
        \skp
        \item For each $X,Y\in\cC$, the Hom-object $\cC^{\cW}(X,Y):=F(\cC^{\cV}(X,Y))$;
         \skp
        \item For all $X,Y,Z\in\cC$, the composition morphism $\gamma^\cW_{X,Y,Z}:=F(\gamma^\cV_{X,Y,Z}) \circ F^{(2)}_{\cC^{\cV}(Y,Z),\cC^{\cV}(X,Y)}$;
         \skp
        \item For each $X\in\cC$, the identity morphisms $\kappa^\cW_X:=F(\kappa^\cV_X) \circ F^{(0)}$. \qed
    \end{enumerate}
\end{lemma}

% SEE NOTE 79, PAGE 2

\subsection{Closed monoidal categories}
\label{sec:prelim-closed} 
We refer the reader to \cite[Section~1.5]{Kelly2005} for further details.

\begin{definition}[\textnormal{$[Y,-]$}]
A monoidal category $\cC=(\cC,\otimes,\one)$ is {\it right-closed } %\red{[CW: I looked into this again and many references use ``right-closed" for this notion include Baez, Lurie, Wikipedia, and by Kelly (see comments). After looking at several references, it seems like 70\% call this right-closed. (It is annoying how this isn't uniform.) We will change to ``right-closed"]}
%http://www.tac.mta.ca/tac/volumes/9/n4/n4.pdf
if for every $Y\in\cC$, the functor $(-\otimes Y):\cC\rightarrow\cC$ admits a right adjoint $[Y,-]:\cC\rightarrow\cC$. For an object $Z \in \cC$, we refer to $[Y,Z] \in \cC$ as the {\it internal Hom of $Y$ and $Z$}.
\end{definition}

Next, we will recall how right-closed monoidal categories are always enriched over themselves. To do so, we borrow the following notation from \cite[Section~4.1]{Riehl2017}.

\begin{notation}[$\nu^\#$, $\psi^\flat$]\label{not:transpose}
If $\cC,\cD$ are arbitrary categories, then any adjunction, \[(F_1: \cC \to \cD) \dashv (F_2: \cD \to \cC),\] gives rise to natural transformations $\eta:\id_\cC \Rightarrow F_2F_1$ and $\epsilon:F_1F_2\Rightarrow\id_\cD$ called the {\it unit} and {\it counit} of the adjunction, respectively, subject to triangle compatibility axioms. The unit and counit induce natural bijections of morphisms for all $C\in\cC$ and $D\in\cD$:
\begin{align*}
\Hom_{\cD}(F_1(C),D) & \; \cong \;  \Hom_{\cC}(C,F_2(D)) \\
(\psi:F_1(C)\to D) &\; \mapsto \;(\psi^\flat:C \stackrel{\eta_C}{\longrightarrow}F_2F_1(C)\stackrel{F_2(\psi)}{\longrightarrow} F_2(D))
\\[.05pc]
(\nu^\#:F_1(C)\stackrel{F_1(\nu)}{\longrightarrow}F_1F_2(D)\stackrel{\epsilon_D}{\longrightarrow}D) &\; \mapsfrom \;(\nu:C\rightarrow F_2(D)).
\end{align*}
We call $\nu^\#$ the {\it left transpose} of $\nu$, and $\psi^\flat$ the {\it right transpose} of $\psi$.
\end{notation}

Using the triangle axioms from the adjunction $F_1 \dashv F_2$, it is straight-forward to check that 
\begin{equation} \label{eq:sharp-flat}
(\nu^\#)^\flat = \nu
\quad \quad \text{and} \quad \quad 
(\psi^\flat)^\# = \psi.
\end{equation}

\begin{example} \label{ex:right-closed}
 If $(\cC,\otimes,\one)$ is a right-closed monoidal category with  adjunction
$(-\otimes Y) \dashv [Y,-]$ for $Y \in \cC$, then  we obtain a bijection,
\begin{equation} \label{eq:Crtclosed}
    \Hom_{\cC}(X\otimes Y, Z)  \; \cong \; \Hom_{\cC}(X, [Y,Z] ),
\end{equation}
which is natural in each variable  $X,Z \in \cC$. The components for the unit and counit of this adjunction are  given by 
\[
(\id_{X \otimes Y})^\flat =: \eta^Y_X: X \to [Y, X \otimes Y]
\quad \quad \text{and} \quad \quad
(\id_{[Y,Z]})^\# =:  \epsilon^Y_Z:[Y,Z] \otimes Y \to Z.
\] 
So, the left transpose of $\nu:X\rightarrow [Y,Z]$, and the right transpose of $\psi:X \otimes Y \rightarrow Z$ in $\cC$, are given respectively by:
\begin{equation} \label{eq:right-closed}
\nu^\#: X\otimes Y \stackrel{\nu\otimes \id_Y}{-\kern-5pt\longrightarrow} [Y,Z]\otimes Y \stackrel{\epsilon^Y_Z}{\longrightarrow} Z
\quad \quad \text{and} \quad \quad
\psi^\flat: X \stackrel{\eta^Y_X}{\longrightarrow} [Y, X \otimes Y] \stackrel{[Y,\psi]}{-\kern-5pt\longrightarrow} [Y,Z].
\end{equation}
As a consequence of (\ref{eq:Crtclosed}), the underlying category $\cC_0$ of $\cC$ simply recovers the original category, due to the correspondences:
\[\Hom_{\cC_0}(X,Y)\;  \cong \; \Hom_{\cC}(\one, [X,Y]) \;  \cong \; 
 \Hom_{\cC}(\one \otimes X,Y)\;  \cong \; 
 \Hom_{\cC}(X,Y).
\]   
\end{example}

The preliminary result below will be of use later.

\begin{proposition}\label{prop:self-enriched}
Every right-closed monoidal category $(\cC, \otimes, \one)$ is enriched over itself by taking: 
\begin{enumerate}[\upshape (a)]
\item The class of objects $X, Y,\dots \in \cC$; 

\skp

\item $\cC(X,Y):= [X,Y]$;

\skp

\vspace{.02in}

\item $\kappa^\cC_X:\one\rightarrow [X,X]$ is the right transpose of  the left unitor $\ell_X:\one\otimes X\stackrel{\sim}{\rightarrow} X$ of $X$ in $\cC$; 

\skp

\item $\gamma^\cC_{X,Y,Z}: [Y,Z]\otimes[X,Y]\rightarrow [X,Z]$ is the right transpose of the morphism below
\[[Y,Z]\otimes[X,Y]\otimes X \stackrel{\id \otimes\epsilon^X_Y}{-\kern-5pt\longrightarrow}[Y,Z]\otimes Y \stackrel{\epsilon^Y_Z}{\longrightarrow} Z.\] 

\vspace{-.25in}

\qed
\end{enumerate} 
\end{proposition}

\subsection{Enriched categories of modules in monoidal categories} \label{sec:enriched-mod} Now we will show that a category of modules in a closed monoidal category $\cC$ is enriched over $\cC$. First, let us consider the following preliminary result on equalizing morphisms in ordinary categories. 

\begin{lemma}\label{lem:equalizers}
Let $F_1\dashv F_2$ be adjoint functors. A morphism $\nu:X\rightarrow Y$ equalizes morphisms $\nu', \nu'': Y\rightarrow F_2(Z)$ if and only if $F_1(\nu):F_1(X)\rightarrow F_1(Y)$ equalizes $(\nu')^\#,(\nu'')^\#: F_1(Y)\rightarrow Z$. 
\end{lemma}

\begin{proof}
Notice that $\nu' \nu=\nu'' \nu$ if and only if $(\nu' \nu)^\#=(\nu'' \nu)^\#$. The latter equality can be expressed  as $\epsilon_Z  F_1(\nu') F_1(\nu)=\epsilon_Z F_1(\nu'')   F_1(\nu)$.  Therefore, $(\nu')^\# F_1(\nu)=(\nu'')^\# F_1(\nu)$.
\end{proof}

%Now we discuss when a monoidal category of modules in a category $\cC$ is a $\cC$-category.

Now we present the main result of the section.

\begin{theorem} \label{thm:closed}
Let $(\cC, \otimes, \one)$ be a right-closed monoidal category and take $(A,m,u)\in\Alg(\cC)$. If $\cC$ has equalizers, then $\mathsf{Mod}(\cC)_A$ is enriched over $\cC$.
\end{theorem}

\begin{proof}
Following Definition~\ref{def:enr-cat}, for each pair of modules $(M,\rho^M)$ and $(N,\rho^N)$ in $\mathsf{Mod}(\cC)_A$, we  need to define a corresponding Hom-object in $\cC$, which we will denote by 
\[
[M,N]_A := \mathsf{Mod}(\cC)_A(M,N),  
\]
and show that there exists composition and identity morphisms for these objects satisfying associative and unitality axioms.
By using the right-closure of $\cC$ and Example~\ref{ex:right-closed}, we define 
\[
(\text{eq}:[M,N]_A\hookrightarrow [M,N]):= \text{the equalizer of the morphisms $R_{M,N}, S_{M,N}:[M,N]\rightarrow[M\otimes A,N]$},
\]
where the left transposes of $R_{M,N}$ and $S_{M,N}$ are given as follows:
\begin{align*}
R_{M,N}^\# &:  [M,N]\otimes M\otimes A\stackrel{\id \otimes\rho^M}{-\kern-5pt\longrightarrow}[M,N]\otimes M\stackrel{\epsilon^M_N}{\longrightarrow} N,\\
S_{M,N}^\# &:  [M,N]\otimes M\otimes A\stackrel{\epsilon^M_N\otimes \id}{-\kern-5pt\longrightarrow}N\otimes A\stackrel{\rho^N}{\longrightarrow} N. 
\end{align*}
In particular, $R_{M,N} = (R_{M,N}^\#)^\flat$ and $S_{M,N} = (S_{M,N}^\#)^\flat$ by \eqref{eq:sharp-flat}. 

\skp
Since $\cC$ has equalizers, we obtain that $[M,N]_A$ is an object of $\cC$, and $\text{eq}: [M,N]_A \hookrightarrow [M,N]$ is a monomorphism in $\cC$. Moreover, by Lemma~\ref{lem:equalizers}, we get that 
\begin{equation} \label{eq:equal-sharp}
 R_{M,N}^\#  \circ  (\text{eq}  \otimes  \id_{M \otimes A}) \; =  \; S_{M,N}^\# \circ (\text{eq}   \otimes  \id_{M \otimes A})
\end{equation}
as morphisms $[M,N]_A \otimes M \otimes A \to N$ in $\cC$. 

\skp

Now the rest of the proof follows from Claims~\ref{claim1} and~\ref{claim2} below. 

\begin{claim} \label{claim1}
\textnormal{(a)}\; For all $A$-modules $M,N,P\in\mathsf{Mod}(\cC)_A$, there is a composition morphism:    
 \[\gamma^A_{M,N,P}:[N,P]_A\otimes[M,N]_A\rightarrow[M,P]_A \]
rendering the following diagram commutative.
\[
{\small
\xymatrix@C=4pc@R=1.2pc{
 & [N,P]_A\otimes[M,N]_A 
    \ar[d]^{\textnormal{eq} \; \otimes \; \textnormal{eq}}
    \ar@{-->}[ldd]_{\gamma^A_{M,N,P}}
 & \\
 & [N,P]\otimes[M,N]
    \ar[d]^{\gamma^\cC_{M,N,P}}
 & \\
[M,P]_A
    \ar[r]^{\textnormal{eq}}
& [M,P]
    \ar@/^/[r]^{R_{M,P}}\ar@/_/[r]_{S_{M,P}}
& [M\otimes A,P]
}
}
 \]
\textnormal{(b)}\; The composition of the morphisms $\gamma^A_{M,N,P}$ is associative, i.e., for all  $M,N,P,Q \in\mathsf{Mod}(\cC)_A$, we get:
    \[ \gamma^A_{M,N,Q} \circ (\gamma^A_{N,P,Q}  \otimes  \id_{[M,N]_A}) \; = \; 
    \gamma^A_{M,P,Q}\circ ( \id_{[P,Q]_A}  \otimes \gamma^A_{M,N,P} ). \]
\end{claim}

\begin{claim} \label{claim2}
\textnormal{(a)} For each $(M,\rho^M)\in\mathsf{Mod}(\cC)_A$, there is an identity morphism $\kappa^A_M:\one\rightarrow[M,M]_A$
rendering the following diagram commutative. 
\[
{\small
\xymatrix@C=4pc{
 & \one
    \ar[d]^{\kappa^\cC_M}
    \ar@{-->}[ld]_{\kappa^A_M}
 & \\
[M,M]_A
    \ar[r]^{\textnormal{eq}}
 & [M,M]
    \ar@/^/[r]^(.45){R_{M,M}}\ar@/_/[r]_(.45){S_{M,M}}
& [M\otimes A,M]
}
}
 \]

\noindent \textnormal{(b)} The morphisms $\kappa^A_M$ satisfy the unitality axioms, i.e. for all $M,N\in\mathsf{Mod}(\cC)_A$, we have:
\[ \gamma^A_{M,N,N} \circ (\kappa^A_N  \otimes  \id_{[M,N]_A}) \; = \; \ell_{[M,N]_A}, \quad  
\quad 
\gamma^A_{M,M,N} \circ (\id_{[M,N]_A} \otimes  \kappa^A_M) \; = \; r_{[M,N]_A}.
\]
\end{claim}

\smallskip

The proofs of Claims~\ref{claim1} and~\ref{claim2} are provided in Appendix~\ref{app:Sec3.3claims}. This concludes the proof of the theorem.
\end{proof}

%%%%%%%%%%%%%%%%%%%%%%%%%%
%%%%%%%%%%%%%%%%%%%%%%%%%%
%%%%%%%%%%%%%%%%%%%%%%%%%%

\subsection{Endomorphism algebras in categories of modules} \label{sec:endomalg} In this part, we will study the endomorphism algebra $[A,A]_A$ arising in the setting of Theorem~\ref{thm:closed}.

\begin{theorem}\label{thm:endomalg}  Let $(\cC, \otimes, \one)$ be a right-closed monoidal category, and assume that $\cC$ has equalizers. If $A\in\mathsf{Alg}(\cC)$, then $A\cong [A,A]_A$ as algebras in $\cC$.  
\end{theorem}

\begin{proof}
First, we verify $[A,A]_A$ is an algebra in $\cC$. It is an object of $\cC$ by Theorem~\ref{thm:closed}. Moreover, the algebra structure maps of $[A,A]_A$ are given by the $\cC$-enrichment. Namely, the multiplication map is given by composition $\gamma^A_{A,A,A}:[A,A]_A\otimes [A,A]_A\rightarrow [A,A]_A$, and the unit map is given by the identity map $\kappa_A^A:\one\rightarrow [A,A]_A$. The associative and unital laws follow from the enrichment axioms.

Next, we construct inverse maps $\varphi:A\rightarrow [A,A]_A$ and $\psi: [A,A]_A \rightarrow A$ in $\cC$. Define $\psi$ as:

\vspace{-.15in} 

\[
\xymatrix@C=3pc{
\psi: [A,A]_A \ar[r]^(.4){\textnormal{eq}}
&[A,A] =[A,A]\otimes\one
\ar[r]^(.58){\id \; \otimes \; u}
&[A,A]\otimes A \ar[r]^(.58){\epsilon_A^A} 
&A.
}
\]

\vspace{.05in}

\noindent To construct $\varphi$, we will need the universal property of the equalizer $[A,A]_A$.

\begin{claim} \label{claim3}
The right transpose $m^\flat:A\rightarrow[A,A]$ of the multiplication map $m$ of $A$ equalizes the maps $R_{A,A}$ and $S_{A,A}$ defined in Theorem~\ref{thm:closed}. Thus, we obtain a map $\varphi$ such that $\textnormal{eq}\circ \varphi=m^\flat$.
\[
{\small
\xymatrix@C=4pc{
 & A
    \ar[d]^(.45){m^\flat}
    \ar@{-->}[ld]_(.45){\varphi}
 & \\
[A,A]_A
    \ar[r]^{\textnormal{eq}}
 & [A,A]
    \ar@/^/[r]^(.45){R_{A,A}}\ar@/_/[r]_(.45){S_{A,A}}
& [A\otimes A,A]
}
}
 \]
 \end{claim}
\noindent The proof of this claim is provided in Appendix~\ref{app:Sec3.4claims}.

%Hence, the universal property of $[A,A]_A$ tells us that $m^\flat$ induces a map $\varphi:A\rightarrow[A,A]_A$ satisfying $\text{eq}\circ\varphi=m^\flat$.

We see that $\varphi$ and $\psi$ are mutual inverses as follows. We compute: 
%First we check $\psi\circ\varphi=\id_A$ via the computation below:
\[
\begin{array}{rll}
    \psi \hspace{0.02in} \varphi &= \epsilon_A^A \hspace{0.02in} (\id_{[A,A]}\otimes u) \hspace{0.02in} \text{eq}\hspace{0.042in} \varphi 
    &= \epsilon_A^A \hspace{0.02in} (\id_{[A,A]}\otimes u) \hspace{0.02in} (m^\flat\otimes\id_\one) \\[.4pc]
    &= \epsilon_A^A \hspace{0.02in} (m^\flat\otimes\id_A) \hspace{0.02in} (\id_A\otimes u) 
    &= \epsilon_A^A \hspace{0.02in} ([A,m]\otimes\id_A) \hspace{0.02in} (\eta_A^A\otimes\id_A) \hspace{0.02in} (\id_A\otimes u) \\[.4pc]
    &= m \hspace{0.02in} \epsilon_{A\otimes A}^A \hspace{0.02in} (\eta_A^A\otimes\id_A) \hspace{0.02in} (\id_A\otimes u) 
    &= m \hspace{0.02in} (\id_A\otimes u) 
    \; \; = \id_A.
\end{array}
\]
The first equality follows by the definition of $\psi$, the second by the construction of $\varphi$, the third by level-exchange, the fourth by definition of $m^\flat$, the fifth by naturality of $\epsilon$, the sixth by an adjunction triangle axiom, and the last by a unitality axiom.
On the other hand, checking that $\varphi \hspace{0.02in} \psi=\id_{[A,A]_A}$ is equivalent to checking that $m^\flat \hspace{0.02in} \epsilon_A^A \hspace{0.02in} (\id_{[A,A]}\otimes u) = \id_{[A,A]}$,
by simply applying the monomorphism  $\text{eq}$ on the left of both sides  and unpacking definitions. This latter equality is proven below.

\vspace{-.2in}

\[
{\small
\begin{tikzcd}[row sep=small, column sep=large]
	{[A,A]} &&& {[A,A]\otimes A} && A \\
	\\
	& {[A,[A,A]\otimes A]} && {[A,[A,A]\otimes A\otimes A]} && {[A,A\otimes A]} \\
	\\
	{[A,A]} &&&&& {[A,A]}
	\arrow["{\eta_{[A,A]}^A}"{pos=0.6}, from=1-1, to=3-2]
	\arrow["{\id\otimes u}", from=1-1, to=1-4]
	\arrow["{\eta_{[A,A]\otimes A}^A}", from=1-4, to=3-4]
	\arrow["{[A,\id\otimes\id\otimes u]}", from=3-2, to=3-4]
	\arrow["{[A,\epsilon_A^A\otimes\id]}", from=3-4, to=3-6]
	\arrow["{\epsilon_A^A}", from=1-4, to=1-6]
	\arrow["{\eta_A^A}", from=1-6, to=3-6]
	\arrow["\id"', from=1-1, to=5-1]
	\arrow["{[A,\epsilon_A^A]}"{pos=0.3}, from=3-2, to=5-1]
	\arrow["\id"', from=5-1, to=5-6]
	\arrow["{[A,m]}", from=3-6, to=5-6]
	\arrow["{[A,\id\otimes u]}"', from=5-1, to=3-6]
\end{tikzcd}
}
\]

\noindent The top two quadrilaterals commute by naturality of $\eta$, the middle quadrangle commutes by level-exchange, the left triangle commutes by an adjunction triangle axiom, and the bottom triangle commutes by a unitality axiom. 

\begin{claim}\label{claim4}
    The map $\varphi:A\rightarrow[A,A]_A$ is unital and multiplicative.
\end{claim}

This claim is proven in Appendix~A.3. Thus, $\varphi$ is an isomorphism of algebras in $\cC$, which concludes the proof of the theorem.
\end{proof}
%%%%%%%%%%%%%%%%%%%%%%%%%%
%%%%%%%%%%%%%%%%%%%%%%%%%%
%%%%%%%%%%%%%%%%%%%%%%%%%%
\subsection{Graded closed monoidal categories} \label{sec:graded-closed}

For the results in this section, recall Hypotheses~\ref{hypotheses}. Specifically, recall that $(\cC,\otimes,\one)$ is assumed to have (possibly infinite) biproducts.

\begin{theorem}[\hbox{$[\![Y,-]\!]$}]
\label{thm:closed2}
If $(\cC,\otimes,\one)$ is a right-closed monoidal category, then so is $(\mathsf{Gr}(\cC),\bar{\otimes},\bar{\one})$. Here, we take Hom objects of $\mathsf{Gr}(\cC)$ to be
\[
[\![Y,Z]\!]:= 
\textstyle \left(\bigoplus_{p \in G} [Y_{g^{-1}p}, Z_p] \right)_{g \in G},
\]
where for a morphism $\varphi$ in $\mathsf{Gr}(\cC)$, we define $[\![Y,\varphi ]\!]_g := \bigoplus_{p\in G}[Y_{g^{-1}p}, \varphi_p]$.
\end{theorem}

\begin{proof}
We wish to show that there is a bijection:
\begin{equation}
\label{eq:baradj}
    \Hom_{\mathsf{Gr}(\cC)}(X \;\bar{\otimes} \; Y, Z)  \; \cong \; \Hom_{\mathsf{Gr}(\cC)}(X, [\![ Y,Z ]\!] ),
\end{equation}
that is natural in each variable $X, Z \in \mathsf{Gr}(\cC)$. To proceed, note that for a morphism $\psi: X \; \bar{\otimes} \; Y \to Z$, the universal property of $\bigoplus$ as coproduct implies that the degree $g$ part of $\psi$ can be captured as a collection of morphisms in $\cC$:
\[
\textstyle \psi_g: \bigoplus_{p \in G} X_p \otimes Y_{p^{-1} g} \to Z_g
 \quad \leftrightsquigarrow \quad 
 \{\psi_{g,p}: X_p \otimes Y_{p^{-1} g} \to Z_g\}_{p \in G}.
\]
Likewise, given a morphism $\nu: X\to [\![Y,Z ]\!]$, the universal property of $\bigoplus$ as product implies that the degree $g$ part of $\nu$ can be captured as a collection of morphisms in $\cC$:
\[
\textstyle \nu_g:  X_g \to \bigoplus_{p \in G} [Y_{g^{-1} p}, Z_p]
 \quad \leftrightsquigarrow \quad 
 \{\nu_{g,p}:X_g \to  [Y_{g^{-1} p}, Z_p]\}_{p \in G}.
\]
Now we define \eqref{eq:baradj} as follows: 
\begin{align*}
\psi: X \; \bar{\otimes} \; Y \to Z \quad &\mapsto \quad
\psi^\flat = \big( \hspace{.02in} (\psi^\flat)_g:=( \textstyle \bigoplus_{p \in G} \;(\psi_{p,g})^\flat )_g \hspace{.02in} \big)_{g \in G} \\[.3pc]
 \nu^\#= \big( \hspace{.02in} (\nu^\#)_g:= (\textstyle \bigoplus_{p \in G} \;(\nu_{p,g})^\# )_g  \hspace{.02in} \big)_{g \in G} \quad &\mapsfrom \quad
\nu: X \to [\![Y, Z]\!]
\end{align*}
Here, the morphisms $(\nu_{p,g})^\#: X_p \otimes Y_{p^{-1}g} \to Z_g$ and $(\psi_{p,g})^\flat: X_g \to [Y_{g^{-1}p}, Z_p]$  exist by the right-closure of $\cC$. It is straight-forward to check that 
\[
(\nu^\#)^\flat = \nu
\quad \quad \text{and} \quad \quad 
(\psi^\flat)^\# = \psi
\]
using the fact that $((\nu_{g,p})^\#)^\flat = \nu_{g,p}$ and
$((\psi_{g,p})^\flat)^\# = \psi_{g,p}$
in $\cC$; see \eqref{eq:sharp-flat}. Moreover, we check naturality in $X$  and leave the computation for $Z$ for the reader. Take a morphism $\chi: X' \to X$. Then for any $\psi \in \Hom_{\mathsf{Gr}(\cC)}(X \;\bar{\otimes} \; Y, Z)$, we obtain that:
\[
[\chi^* \circ (-)^\flat](\psi) \; = \;[(-)^\flat\circ (\chi \; \bar{\otimes} \; \id_{Y})^*](\psi) \quad \text{in}  \; \Hom_{\mathsf{Gr}(\cC)}(X', [\![Y, Z ]\!]).
\]
Indeed, in degree $g$ we have the following computation:
\begin{align*}
(\psi^\flat \circ \chi)_g
& \; = \; \textstyle \bigoplus_{p \in G}\;  (\psi_{p,g})^\flat \circ \chi_g 
 \; \overset{(*)}{=} \;
\textstyle \bigoplus_{p \in G}\;  (\psi_{p,g}\circ (\chi_g \otimes \id_{\hspace{.01in}Y_{g^{-1}p}}))^\flat \\[.2pc]
& \; = \;
\textstyle \bigoplus_{p \in G}\;  ((\psi \circ (\chi\; \bar{\otimes} \; \id_{\hspace{.01in}Y}))_{p,g})^\flat 
 \; = \;
((\psi \circ (\chi\; \bar{\otimes} \; \id_{\hspace{.01in}Y}))^\flat)_{g}.
\end{align*}
Here, we use  the right-closure of $\cC$, i.e., the naturality in $X$ for the bijection \eqref{eq:Crtclosed}, to justify~$(*)$. 
\end{proof}

Combining Theorem~\ref{thm:closed} and~\ref{thm:closed2} above, we obtain the following result.

\begin{corollary} \label{cor:closed} 
Let $(\cC,\otimes,\one)$ be a right-closed monoidal category with equalizers, and take an algebra $A$ in $\mathsf{Gr}(\cC)$. Then, the category of modules $\mathsf{GrMod}(\cC)_A$ is enriched over $\mathsf{Gr}(\cC)$.
\end{corollary}

\begin{proof}
Our assumptions ensure that $(\mathsf{Gr}(\cC),\bar{\otimes},\bar{\one})$ is right-closed and has equalizers  by Theorem~\ref{thm:closed2}.
%and Lemma~\ref{lem:bar-identities}(c).
Indeed, $\mathsf{Gr}(\cC)$ inherits any categorical co/limits (e.g., equalizers) that $\cC$ has.
Now apply Theorem~\ref{thm:closed} to the monoidal category $\mathsf{Gr}(\cC)$ and to the  algebra $A$ in $\mathsf{Gr}(\cC)$ to obtain the result.
\end{proof}

Now applying Lemma~\ref{lem:tranport-enrichment} to the monoidal functor $\langle -\rangle_e:\mathsf{Gr}(\cC)\rightarrow\cC$ from Proposition~\ref{prop:e-functor-monoidal}, we get the immediate result below.

\begin{corollary} \label{cor:GrModCA-main} 
Let $(\cC,\otimes,\one)$ be a right-closed monoidal category with equalizers, and take an algebra $A$ in $\mathsf{Gr}(\cC)$. Then,  $\mathsf{GrMod}(\cC)_A$ is enriched over $\cC$, with Hom objects
    \[(\mathsf{GrMod}(\cC)_A)^{\cC}(M,N):= \langle[\![M,N]\!]_A\rangle_e,\]
for $M,N \in \mathsf{GrMod}(\cC)_A$.
\qed
\end{corollary}

%%%%%%%%%%%%%%%%%%%%%%%%%%
%%%%%%%%%%%%%%%%%%%%%%%%%%
%%%%%%%%%%%%%%%%%%%%%%%%%%

\subsection{Graded categories with shifts} \label{sec:graded-shifts}
In this part, we generalize Zhang's notion of a $G$-graded category \cite[Definition~3.2]{Zhang1996} to the enriched setting.

\begin{definition}[$G$-graded category]\label{def:G-grading}
    Let $\cD$ be a $\cV$-category. A \textit{$G$-grading} on $\cD$ consists of a family of $\cV$-enriched functors $(S_g:\cD\rightarrow\cD,\{(S_g)_{X,Y}:\cD(X,Y)\rightarrow\cD(S_g(X),S_g(Y))\})_{g\in G}$ satisfying the following conditions, for all $X,Y\in\cD$ and all $g,h\in G$.
    \begin{enumerate}[(a)]
        
        \item $S_e(X)=X$.
        
        \skp
        
        \item $(S_e)_{X,Y}=\id_{\cD(X,Y)}$.
        
        \skp

        \item $S_g(S_h(X))=S_{gh}(X)$.

        \skp
        
        \item $(S_g)_{S_h(X),S_h(Y)}(S_h)_{X,Y}=(S_{gh})_{X,Y}$.

        \skp
        
        \item $(S_g)_{X,Y}$ are isomorphisms.
    \end{enumerate}
If $\cD$ is equipped with a $G$-grading, we call it a \textit{$G$-graded $\cV$-category}. The enriched functors $(S_g,\{(S_g)_{X,Y})\})$ are called the \textit{shift functors} of $\cD$.
\end{definition}

For the next result, consider $\mathsf{GrMod}(\cC)_A$ as a $\cC$-category as in Corollary \ref{cor:GrModCA-main}.

\begin{proposition} \label{prop:GrModCA-gr}
We have that $\mathsf{GrMod}(\cC)_A$ is a $G$-graded $\cC$-category, where:
\begin{enumerate}[\upshape (a)]
\item $S_g:(M,\rho^M)\mapsto(S_g(M),\rho^{S_g(M)})$ with $S_g(M):=(M_{g^{-1}d})_{d\in G}$ and $\rho^{S_g(M)}:=(\rho^M_{g^{-1}d})_{d\in G}$,
 \skp
\item $(S_g)_{M,N}:=\id_{\langle[\![M,N]\!]_A\rangle_e}:\langle[\![S_g(M),S_g(N)]\!]_A\rangle_e\rightarrow\langle[\![M,N]\!]_A\rangle_e$, 
\end{enumerate}
\smallskip

\noindent for all $(M,\rho^M),(N,\rho^N)\in\mathsf{GrMod}(\cC)_A$.
\end{proposition}
\begin{proof}
    It is immediate from the definitions that the shift functors satisfy conditions (a) and (c) of Definition~\ref{def:G-grading}. If $(S_g)_{M,N}=\id_{\langle[\![M,N]\!]_A\rangle_e}$, then it is also clear that the shift functors $\big(S_g,\{(S_g)_{M,N}\}\big)$ form $\cC$-functors and that they satisfy the remaining conditions of Definition~\ref{def:G-grading}. It remains to show that $\langle[\![S_g(M),S_g(N)]\!]_A\rangle_e=\langle[\![M,N]\!]_A\rangle_e$, so that we may indeed define $(S_g)_{M,N}$ as the identity map. This last step follows from part (c) of the auxiliary result below, in the case where $d=e$.

%See Note #99 for complete proof
\begin{claim}\label{claim:S_g-props}
    For all $M,N\in \mathsf{GrMod}(\cC)_A$ and $g\in G$:
    \begin{enumerate}[\upshape(a)]
        \item $S_g([\![M,N]\!]_A) = [\![M,S_g(N)]\!]_A$;
        \skp
        \item $\left([\![S_g(M),N]\!]_A \right)_d = \left([\![M,N]\!]_A\right)_{dg}$, for all $d\in G$;
        \skp
        \item $\left([\![S_g(M),S_g(N)]\!]_A\right)_d = \left([\![M,N]\!]_A\right)_{g^{-1}dg}$, for all $d\in G$.
    \end{enumerate}
\end{claim} 
\begin{proof}[Proof of Claim~\ref{claim:S_g-props}.]
    We give a short proof of part (a). Parts (b) and (c) follow by similar arguments. Recall that over the course of the proof of Theorem~\ref{thm:closed}, we defined $[\![M,N]\!]_A$ as the equalizer of two maps $R_{M,N},S_{M,N}:[\![M,N]\!]\rightarrow[\![M\bar{\otimes}A,N]\!]$, given by:
    \begin{align*}
        R_{M,N} &:= \text{the right transpose of } \epsilon^M_N(\id_{[\![M,N]\!]}\; \bar{\otimes}\; \rho^M);\\[.2pc]
        S_{M,N} &:= \text{the right transpose of } \rho^N(\epsilon^M_N \; \bar{\otimes}\; \id_A).       
    \end{align*}
It follows that $S_g([\![M,N]\!]_A)$ is the equalizer of the maps $S_g(R_{M,N}):=((R_{M,N})_{g^{-1}d})_{d\in G}$ and $S_g(S_{M,N}):=((S_{M,N})_{g^{-1}d})_{d\in G}$.
Similarly $[\![M,S_g(N)]\!]_A$ is defined as the equalizer of the pair of maps $R_{M,S_g(N)},S_{M,S_g(N)}:[\![M,S_g(N)]\!]\rightarrow[\![M\bar{\otimes}A,S_g(N)]\!]$, where:
    \begin{align*}
        R_{M,S_g(N)} &:= \text{the right transpose of } \epsilon^M_{S_g(N)}(\id_{[\![M,S_g(N)]\!]}\; \bar{\otimes}\; \rho^M);\\[.2pc]
        S_{M,S_g(N)} &:= \text{the right transpose of } \rho^{S_g(N)}(\epsilon^M_{S_g(N)}\; \bar{\otimes}\; \id_A).        
    \end{align*}
A quick computation shows that: 
\[S_g([\![M,N]\!]) =[\![M,S_g(N)]\!]\quad \text{ and }\quad S_g([\![M\bar{\otimes}A,N]\!]) =[\![M\bar{\otimes}A,S_g(N)]\!], \]
so $S_g(R_{M,N}), S_g(S_{M,N})$ and $R_{M,S_g(N)},S_{M,S_g(N)}$ share the same domain and codomain. More computations show that for all degrees $d\in G$, we have:
    \begin{align*}
        (S_g(R_{M,N}))_d &= (R_{M,N})_{g^{-1}d} = (R_{M,S_g(N)})_d,\\[.1pc]
        (S_g(S_{M,N}))_d &= (S_{M,N})_{g^{-1}d} = (S_{M,S_g(N)})_d.        
    \end{align*}
Therefore, $S_g([\![M,N]\!]_A)$ and $[\![M,S_g(N)]\!]_A$ are equalizers of the same pair of maps.
\end{proof}
This completes the proof of the proposition.
\end{proof}

\subsection{Zhang-Morita equivalence} \label{sec:ZM-equiv}

 Next, we introduce the main notion of sameness for graded algebras in monoidal categories that is the focus of this article. Recall Corollary~\ref{cor:GrModCA-main}.

\begin{definition} \label{def:ZM-equiv}
Two algebras $A$ and $B$ in $\mathsf{Gr}(\cC)$
are said to be {\it Zhang-Morita equivalent} (or, {\it ZM-equivalent}) 
if $\mathsf{GrMod}(\cC)_A \simeq^{\cC} \mathsf{GrMod}(\cC)_B$.
\end{definition}

\begin{remark} \label{rem:Morita} The name above is due to the work of K. Morita \cite{Morita1958} on the equivalence of categories of modules over rings, and  the work of J. Zhang \cite{Zhang1996} whose work on categories of graded modules over $\Bbbk$-algebras, for $\Bbbk$ a field,  we are generalizing in this article. We discuss the differences between ZM-equivalence and similar notions in other parts of the literature below.
\begin{enumerate}[\upshape (a)]

\item Two algebras $A$ and  $B$ in a monoidal category $(\cC,\otimes,\one)$ are  {\it Morita equivalent} if $\mathsf{Mod}(\cC)_A$ and $\mathsf{Mod}(\cC)_B$ are equivalent as left $\cC$-module categories; see \cite[Sections~7.1,~7.8]{EGNO2015}. Here, $\mathsf{Mod}(\cC)_A$ is a left $\cC$-module category via action $X \act_A (M,\rho) := (X \otimes M, \id_X \otimes \rho)$. 
The extra structure of a $\cC$-module category is needed to get a rich Morita theory generalizing that for rings and $\Bbbk$-algebras, but this is not used in our work here. 

\skp

\item On the other hand, two rings  are {\it graded Morita equivalent} if their categories of modules, with morphisms in any degree, are equivalent. See, e.g., work of Boisen~\cite{Boisen1994}.  But here, 
 morphisms between graded objects are always degree-preserving [Definition~\ref{def:barC}]; this condition is also imposed in \cite{Zhang1996}.

 \skp

\item  We recover the setting of \cite{Zhang1996} precisely when $(\cC, \otimes, \one)$ is the monoidal category of \linebreak $\Bbbk$-vector spaces $(\mathsf{Vec}_\Bbbk, \otimes_\Bbbk, \Bbbk)$. Indeed, $\mathsf{GrMod}(\mathsf{Vec}_\Bbbk)_A$ is automatically $\mathsf{Vec}_\Bbbk$-enriched.
\end{enumerate}
\end{remark}

%%%%%%%%%%%%%%%%%%%%%%%%%%
%%%%%%%%%%%%%%%%%%%%%%%%%%
%%%%%%%%%%%%%%%%%%%%%%%%%%
\section{Twisted algebraic structures in monoidal categories}
\label{sec:twist}

In this part, we generalize numerous notions in \cite[Section~2]{Zhang1996} on twisting systems and twisted algebras  to the monoidal setting. We introduce twisting systems in monoidal categories in Section~\ref{sec:twistsys}, and then twist algebras and modules in monoidal categories in Sections~\ref{sec:twistedalg} and~\ref{sec:twistedmod},  respectively. To proceed, recall the notation and set-up of Section~\ref{sec:gralg-monoidal}, especially Convention~\ref{conv:algs}, and fix an algebra $A$ in $\mathsf{GrAlg}(\cC) \cong \mathsf{Alg}(\mathsf{Gr}(\cC))$. 

\subsection{Twisting systems} \label{sec:twistsys} 
 In \cite[Section~2]{Zhang1996}, Zhang defines a twisting system on a $G$-graded $\Bbbk$-algebra $A$ as a set of graded $\Bbbk$-linear automorphisms $\tau=\{\tau_g:A\to A\}_{g\in G}$ satisfying the twisting condition $\tau_g(y\tau_h(z)) = \tau_g(y)\tau_{gh}(z)$ for all $g,h,\ell\in G$ and $y\in A_h$, $z\in A_\ell$.
Analogously, we define the following notion.

\begin{definition} \label{def:twist}
A {\it twisting system} for $A$ (or a {\it twist} of $A$) is a tuple of isomorphisms,
\[
\tau:=(\tau_d(g): A_g \overset{\sim}{\longrightarrow} A_g)_{d,g \in G},
\]
satisfying the {\it twisting condition} given in the commutative diagram below. 
\begin{equation} \label{eq:twistcond}
\xymatrix@R=1.5pc@C=5pc{
A_{g_1} \otimes A_{g_2} 
\ar[r]^{\tau_d(\textnormal{$g_1$}) \; \otimes \; \id}
\ar[d]_{\id \; \otimes \; \tau_{\textnormal{$g_1$}}(\textnormal{$g_2$})}
& 
A_{g_1} \otimes A_{g_2}
\ar[r]^{\id \; \otimes \; \tau_{\textnormal{$d \hspace{.01in} g_1$}}(\textnormal{$g_2$})}
& 
A_{g_1} \otimes A_{g_2} 
\ar[d]^{m_{\textnormal{$g_1,g_2$}}}\\
A_{g_1} \otimes A_{g_2} 
\ar[r]_{m_{\textnormal{$g_1,g_2$}}}
&
A_{g_1 g_2} 
\ar[r]_{\tau_{d}(\textnormal{$g_1 g_2$})}
&
A_{g_1 g_2} 
}
\end{equation}
We will refer to the diagram above with respective entries $d, g_1, g_2$ by ``\eqref{eq:twistcond}$(d,g_1,g_2)$''.
\end{definition}

\begin{example} \label{ex:identwistsys}
We have an identity twisting system: $\tau^{\textnormal{id}}_d(g) := \id_{A_g}$ for all $d,g \in G$.
\end{example}

\begin{example} \label{ex:2cocycle-twist}
Consider the commutative monoid $\mathsf{Aut}_\cC(\one)$ (see \cite[Section~1.3]{TuraevVirelizier2017}), and  take a two-cocycle $\alpha$ of $G$ with values in $\mathsf{Aut}_\cC(\one)$. That is, $\alpha: G \times G \to \mathsf{Aut}_\cC(\one)$ is a map that satisfies 
\[
\alpha(gh, \ell)   \alpha(g,h) \; = \; \alpha(g,h\ell) \alpha(h, \ell) 
 \] 
 for all $g, h, \ell \in G$. Then the maps below, for each $d,g \in G$,
 \[
 \xymatrix@C=2pc{
\tau_d(g):=\alpha(d,g) \otimes  \id: A_g = \one \otimes A_g
\ar[r]
&\one \otimes A_g  = A_g
}
 \]
 satisfies the twisting condition \eqref{eq:twistcond}.
\end{example}

We will need the useful result below.

\begin{lemma} \label{lem:tau-inv}
 We have that $\tau_g(e)^{-1}u_e = \tau_e(e)^{-1}u_e$, for all $g \in G$.
\end{lemma}

\begin{proof} Consider the following computations:
\begin{equation}
\begin{aligned}
m_{e,e} \left(\tau_e(e)u_e \otimes \tau_e(e)^{-1}u_e\right) 
&\; = \;    
m_{e,e} \hspace{0.02in} (\tau_e(e) \otimes \id_{A_e})\; (\id_{A_e} \otimes \tau_e(e)^{-1}) \; (u_e \otimes u_e)\\[-.1pc]
&\; \overset{\eqref{eq:twistcond}(e,e,e)}{=} \; 
\tau_e(e)\; m_{e,e} \; (\id_{A_e} \otimes \tau_e(e)^{-1})\; (u_e \otimes u_e)\\[.1pc]
&\; = \; 
\tau_e(e)\; m_{e,e}\; (u_e \otimes \id_{A_e})\; \tau_e(e)^{-1}\; u_e\\[.1pc]
&
\; = \;
\tau_e(e)\; \tau_e(e)^{-1}\; u_e
\; = \;
u_e;
\end{aligned}
\label{eq:tau-inv1}
\end{equation}

\vspace{-.05in} 

\begin{equation}
\begin{aligned}
m_{e,e} \left(\tau_g(e)^{-1}u_e \otimes \tau_e(e)u_e\right) 
&\; = \;    
m_{e,e} \hspace{0.02in} (\id_{A_e} \otimes \tau_e(e))\; (\tau_g(e)^{-1} \otimes \id_{A_e})\; (u_e \otimes u_e)\\[-.1pc]
&\; \overset{\eqref{eq:twistcond}(g,e,e)}{=} \; 
\tau_g(e)^{-1} \;  m_{e,e}\;(\id_{A_e} \otimes \tau_g(e))\;(u_e \otimes u_e)\\[.1pc]
&\; = \; 
\tau_g(e)^{-1}\; m_{e,e}\; (u_e \otimes \id)\; \tau_g(e)\; u_e\\[.1pc]
&
\; = \;
\tau_g(e)^{-1}\; \tau_g(e)\; u_e
\; = \;
u_e.
\end{aligned}
\label{eq:tau-inv2}
\end{equation}
Here, we also employ level-exchange and unitality axioms. Now the result holds because: 
\begin{align*}
\tau_g(e)^{-1}\;u_e
&\; = \; 
m_{e,e}\left(\tau_g(e)^{-1} \;u_e \otimes u_e\right)\\[-.1pc]
&\; \overset{\eqref{eq:tau-inv1}}{=} \; 
m_{e,e}\left[\tau_g(e)^{-1} \;u_e \otimes m_{e,e}\left(\tau_e(e) \;u_e \otimes \tau_e(e)^{-1}\;u_e\right)\right]\\[.1pc]
&\; = \; 
m_{e,e}\left[m_{e,e}\left(\tau_g(e)^{-1} \;u_e \otimes \tau_e(e) \;u_e\right) \otimes \tau_e(e)^{-1}\; u_e\right]\\[-.1pc]
&\; \overset{\eqref{eq:tau-inv2}}{=} \; 
m_{e,e}\left(u_e \otimes \tau_e(e)^{-1}\;u_e\right)\\[.1pc]
&\; = \; \tau_e(e)^{-1}\;u_e.
\end{align*}
Here, we use unitality axioms in the first and last equations, and the associativity axiom in the third equation.
\end{proof}

%%%%%%%%%%%%%%%%%%%%%%%%%%
%%%%%%%%%%%%%%%%%%%%%%%%%%
%%%%%%%%%%%%%%%%%%%%%%%%%%

\subsection{Twisted algebras in monoidal categories}  \label{sec:twistedalg}

Next, we use twisted systems to build new algebras in monoidal categories from old.

\begin{definitionproposition}[$A^\tau$] \label{defprop:Atau}
Given a twisting system $\tau$ for $A \in \mathsf{Alg}(\mathsf{Gr}(\cC))$, we obtain that 
\[
(\{A_g\}_g, \quad \{m^\tau_{g,h}:= m_{g,h}  (\id_{A_g} \otimes \tau_g(h)): A_g \otimes A_h \to A_{gh}\}_{g,h}, \quad u_e^\tau := \tau_e(e)^{-1}  u_e: \one \to A_e)
\]
defines an algebra in $\mathsf{Gr}(\cC)$. We  call this algebra the \textnormal{twist of $A$ by $\tau$ in $\mathsf{Gr}(\cC)$}, and denote it by $A^\tau$.
\end{definitionproposition}

\begin{proof}
Indeed, we have that $m^\tau$ is associative due to the computation below:
\[
\begin{array}{rl}
m^\tau_{gh,\ell}\;(m^\tau_{g,h} \otimes \id_{A_\ell})  &= m_{gh,\ell}\;(\id_{A_{gh}} \otimes \tau_{gh}(\ell))\;(m_{g,h} \otimes \id_{A_\ell}) \;(\id_{A_g} \otimes \tau_g(h) \otimes \id_{A_\ell}) \medskip \\
 &= m_{gh,\ell}\;(m_{g,h} \otimes \id_{A_\ell})\;(\id_{A_{g}} \otimes \id_{A_h} \otimes \tau_{gh}(\ell)) \;(\id_{A_g} \otimes \tau_g(h) \otimes \id_{A_\ell}) \medskip \\
  &= m_{g,h\ell}\;(\id_{A_g} \otimes m_{h,\ell})\;(\id_{A_{g}} \otimes \id_{A_h} \otimes \tau_{gh}(\ell)) \;(\id_{A_g} \otimes \tau_g(h) \otimes \id_{A_\ell}) \medskip \\
&= m_{g,h\ell}\;(\id_{A_{g}} \otimes \tau_{g}(h\ell))\;(\id_{A_g} \otimes m_{h,\ell}) \;(\id_{A_g} \otimes \id_{A_h} \otimes \tau_h(\ell)) \medskip\\
&= m^\tau_{g,h\ell}\; (\id_{A_g} \otimes m^\tau_{h,\ell}).
\end{array}
\]
The first and last equations hold by the definition of $m^\tau$; the second equation holds by level-exchange; the third equation follows from  associativity; and the fourth equation follows from the twisting condition \eqref{eq:twistcond}$(g,h,\ell)$.

We have that the unitality axioms hold due to the computations below:
\begin{align*}
m^\tau_{g,e} \;(\id_{A_g} \otimes u^\tau_e) 
&\; = \; m_{g,e}\;(\id_{A_g} \otimes \tau_g(e))\;(\id_{A_g} \otimes \tau_e(e)^{-1})\;(\id_{A_g} \otimes u_e)\\[-.1pc]
&\; \overset{\textnormal{Lem.~\ref{lem:tau-inv}}}{=} \; m_{g,e}\;(\id_{A_g} \otimes \tau_g(e))\;(\id_{A_g} \otimes \tau_g(e)^{-1})\;(\id_{A_g} \otimes u_e)\\[.1pc]
&\; = \; m_{g,e}\;(\id_{A_g} \otimes u_e) \; = \; \id_{A_g};\\[.6pc]
m^\tau_{e,g}\;(u^\tau_e \otimes \id_{A_g}) 
&\; = \; m_{e,g}\;(\id_{A_e} \otimes \tau_e(g))\;(\tau_e(e)^{-1} \otimes \id_{A_g})\;(u_e \otimes \id_{A_g})\\[-.1pc]
&\; \overset{\eqref{eq:twistcond}(e,e,g)}{=} \; \tau_e(g)^{-1}\;m_{e,g}\;(\id_{A_e} \otimes \tau_e(g))\;(u_e \otimes \id_{A_g}) \\[.1pc]
&\; = \; \tau_e(g)^{-1}\;m_{e,g}\;(u_e \otimes \id_{A_g})\; (\id_{A_e} \otimes \tau_e(g))\\[.1pc]
&\; = \; \tau_e(g)^{-1}\;\tau_e(g) \; = \; \id_{A_g}.
\end{align*}
We also employ level-exchange and unitality axioms in the calculations above.
\end{proof}

\begin{example}
Let us continue Example~\ref{ex:2cocycle-twist}. Take a two-cocycle $\alpha: G \times G \to \mathsf{Aut}_\cC(\one)$ of $G$. Then we can twist $A$ via $\alpha$ to form a new algebra $A^{\alpha}$ in $\cC$. Here, the twisted multiplication  and twisted unit are given by 
\[
m^\alpha_{g,h} := \alpha(g,h) m_{g,h} \quad \textnormal{and} \quad u^\alpha_{e} := \alpha(e,e)^{-1}  u_e.
\]
\end{example}

\begin{example} \label{ex:twistalgs}  We discuss below the identity twist, the inverse twist, and the composite twists of algebras in monoidal categories.
\begin{enumerate}[\upshape (a)]
\item For the twisting system $\tau^{\textnormal{id}}$ from Example~\ref{ex:identwistsys}, we get that $A^{\tau^{\textnormal{id}}} = A$ as algebras in $\mathsf{Gr}(\cC)$.

\smallskip

\item Given a twisting system $\tau$ on an algebra $A$, we have that $\tau^{-1}$ is a twisting system on $A^{\tau}$. Indeed, for all $d, g_1, g_2 \in G$ we obtain that:
\begin{align*}
&m^{\tau}_{g_1,g_2}\;(\id_{A_{g_1}} \otimes \tau^{-1}_{d \hspace{.02in} g_1}(g_2))\;(\tau^{-1}_d(g_1) \otimes \id_{A_{g_2}})\\[.1pc]
&\; = \; 
m_{g_1,g_2}\;(\id_{A_{g_1}} \otimes \tau_{g_1}(g_2))\;(\id_{A_{g_1}} \otimes \tau^{-1}_{d \hspace{.02in} g_1}(g_2))\;(\tau^{-1}_d(g_1) \otimes \id_{A_{g_2}})
\\[-.1pc]
&\; \overset{\eqref{eq:twistcond}}{=} \; 
\tau^{-1}_d(g_1 g_2)\; m_{g_1,g_2}\\[.1pc]
&\; = \; 
\tau^{-1}_d(g_1 g_2) \;m_{g_1,g_2}\;(\id_{A_{g_1}} \otimes \tau_{g_1}(g_2))\;(\id_{A_{g_1}} \otimes \tau^{-1}_{g_1}(g_2))\\[.1pc]
&\; = \; 
\tau^{-1}_d(g_1 g_2) \;m_{g_1,g_2}^{\tau}\;(\id_{A_{g_1}} \otimes \tau^{-1}_{g_1}(g_2)).
\end{align*}
In this case, $(A^{\tau})^{\tau^{-1}} = A$ as algebras in $\mathsf{Gr}(\cC)$.
\smallskip

\item Given a twisting system $\tau$ on an algebra $A$ and a twisting system $\sigma$ on $A^\tau$, we have that $\tau \sigma $ is a twisting system on $A$. Indeed, for all $d, g_1, g_2 \in G$ we obtain that:
 \begin{align*}
 \hspace{.3in}
&m_{g_1,g_2} \; [\id_{A_{g_1}} \otimes \tau_{d \hspace{.02in} g_1}(g_2) \; \sigma_{d \hspace{.02in} g_1}(g_2)] \; [\tau_d(g_1) \; \sigma_d(g_1) \otimes \id_{A_{g_2}}]\\[.1pc]
&\; = \; m_{g_1,g_2} \;  [\id_{A_{g_1}} \otimes \tau_{d \hspace{.02in} g_1}(g_2) \; \tau_{g_1}(g_2)^{-1} \; \tau_{g_1}(g_2)  \; \sigma_{d \hspace{.02in} g_1}(g_2)] \; [\tau_d(g_1) \; \sigma_d(g_1) \otimes \id_{A_{g_2}}]\\[.1pc]
&\; = \; m_{g_1,g_2} \;  [\tau_d(g_1)  \otimes \tau_{d \hspace{.02in} g_1}(g_2) \; \tau_{g_1}(g_2)^{-1} ] \; [\sigma_d(g_1) \otimes  \tau_{g_1}(g_2)  \; \sigma_{d \hspace{.02in} g_1}(g_2)]\\[-.1pc]
&\; \overset{\eqref{eq:twistcond}}{=} \;   \tau_d(g_1 g_2) \; m_{g_1,g_2} \; [\sigma_d(g_1) \otimes  \tau_{g_1}(g_2)  \; \sigma_{d \hspace{.02in} g_1}(g_2)]\\[.1pc]
&\; = \;   \tau_d(g_1 g_2) \; m_{g_1,g_2}^{\tau} \; [\sigma_d(g_1) \otimes  \sigma_{d \hspace{.02in} g_1}(g_2)]\\[.1pc]
&\; \overset{\eqref{eq:twistcond} \textnormal{ for $\sigma$}}{=}  \;   \tau_d(g_1 g_2) \; \sigma_d(g_1 g_2) \; m_{g_1,g_2}^{\tau} \; [\id_{A_{g_1}} \otimes  \sigma_{g_1}(g_2)]\\[.1pc]
&\; = \; \tau_d(g_1 g_2) \;\sigma_d(g_1 g_2) \; m_{g_1,g_2} \; [\id_{A_{g_1}} \otimes \tau_{g_1}(g_2) \; \sigma_{g_1}(g_2)].
\end{align*}
In this case, $(A^{\tau})^{\sigma} = A^{\tau \sigma}$ as algebras in $\mathsf{Gr}(\cC)$.
\end{enumerate}
\end{example}

Now consider the following notion of sameness for graded algebras in $\cC$.

\begin{definition} \label{def:twistequiv}
 We say that algebras $A$ and $B$ in $\mathsf{Gr}(\cC)$ are {\it twist equivalent} if there exists a twisting system $\tau$ on $A$ such that $A^\tau \cong B$ as algebras in $\mathsf{Gr}(\cC)$.   
\end{definition}

In fact, given Example~\ref{ex:twistalgs} above,  the result below is immediate.

\begin{proposition}
Twist equivalence in $\mathsf{Alg}(\mathsf{Gr}(\cC))$ is an equivalence relation. \qed
\end{proposition}

We end this part by providing a characterization of twisted algebras, which generalizes \cite[Proposition~2.8]{Zhang1996}.

\begin{proposition} 
\label{prop:phi}
Take $A, B \in \mathsf{Alg}(\mathsf{Gr}(\cC))$. Then the following statements are equivalent.
\begin{enumerate}[\upshape(a)]
\item We have that $A^\tau \cong B$  as algebras in $\mathsf{Gr}(\cC)$, for some twisting system $\tau$ for $A$.
\skp
\item There exist isomorphisms in $\cC$,
\[
\{\phi_d(g): B_g \overset{\sim}{\longrightarrow} A_g\}_{d,g \in G},
\]
that satisfy the commutative diagrams below.
\begin{equation}
\label{eq:phi}
\xymatrix@R=1.5pc@C=8pc{
B_{g_1} \otimes B_{g_2} 
\ar[r]^{\phi_{d}(\textnormal{$g_1$}) \; \otimes \; \phi_{\textnormal{$d \hspace{0.01in} g_1$}}(\textnormal{$g_2$})}
\ar[d]_{m^B_{\textnormal{$g_1,g_2$}}}
&A_{g_1} \otimes A_{g_2} 
\ar[d]^{m^A_{\textnormal{$g_1,g_2$}}} \\
B_{g_1 g_2} \ar[r]_{\phi_{d}(\textnormal{$g_1 g_2$})}& A_{g_1 g_2} 
} \quad  \quad 
\xymatrix@R=1.5pc@C=4pc{
\one 
\ar[r]^{u_e^B}
\ar[rd]_{u_e^A}
& B_e
\ar[d]^{\phi_e(e)} \\
& A_e 
}
\end{equation}
We will refer to the diagram above with respective entries $d, g_1, g_2$ by ``\eqref{eq:phi}$(d,g_1,g_2)$''.
\end{enumerate}
\end{proposition}

\begin{proof} 
(a) $\Rightarrow$ (b).  Suppose that $\varphi: B \to A^\tau$ is an isomorphism  in $\mathsf{Alg}(\mathsf{Gr}(\cC))$. In degree $gh$, we get: 
\begin{equation} \label{eq:phi1}
\varphi_{gh}\; m_{g,h}^B \; =  \; (m_{g,h}^A)^\tau \; (\varphi_{g} \otimes \varphi_h)
\; =  \; 
m_{g,h}^A \; (\id_{A_g} \otimes \tau_g(h)) \; (\varphi_{g} \otimes \varphi_h).
\end{equation}
Now take the isomorphisms in $\cC$:
\[\{\phi_d(g):= \tau_d(g) \; \varphi_g\}_{d,g\in G}.\] 
Then these morphisms satisfy the first condition of \eqref{eq:phi} as follows:
\begin{align*}
 \phi_d(g_1g_2)\; m_{g_1,g_2}^B 
 & \; = \;
 \tau_d(g_1g_2) \; \varphi_{g_1g_2} \;   m_{g_1,g_2}^B \\[-.1pc]
  & \; \overset{\eqref{eq:phi1}}{=} \;
 \tau_d(g_1g_2) \; m_{g_1,g_2}^A \; (\id_{A_{g_1}} \otimes \tau_{g_1}(g_2)) \; (\varphi_{g_1} \otimes \varphi_{g_2})\\[-.1pc]
 & \; \overset{\eqref{eq:twistcond}}{=} \;
m_{g_1,g_2}^A \; (\id_{A_{g_1}} \otimes \tau_{d \hspace{0.02in} g_1}(g_2)) \; (\tau_d(g_1) \otimes \id_{A_{g_2}}) \; (\varphi_{g_1} \otimes \varphi_{g_2})\\[.1pc]
 & \; = \;
m_{g_1,g_2}^A \; (\phi_d(g_1) \otimes \phi_{d \hspace{0.02in} g_1}(g_2)).
\end{align*}
Moreover, the second condition of \eqref{eq:phi} holds as follows:
\[
\phi_e(e) \hspace{0.02in} u_e^B \; =\; \tau_e(e) \hspace{0.02in} \varphi_e \hspace{0.02in} u_e^B \;= \;\tau_e(e) \hspace{0.02in} u_e^{A^\tau} \;= \;\tau_e(e)\hspace{0.02in} \tau_e(e)^{-1} \hspace{0.02in} u_e^A\; =\; u_e^A.
\]

(b) $\Rightarrow$ (a). Suppose that we have isomorphisms $\{\phi_d(g): B_g \overset{\sim}{\to} A_g\}_{d,g\in G}$ satisfying \eqref{eq:phi}. Now take the isomorphisms in $\cC$:
\[
\{\tau_d(g):=\phi_d(g) \; \phi_e(g)^{-1}: A_g \overset{\sim}{\to} A_g\}_{d,g \in G}.
\]
Then these morphisms form a twisting system for $A$, that is, they satisfy \eqref{eq:twistcond}, as we see below:
\begin{align*}
&m^A_{g_1,g_2} \;(\id_{A_{g_1}} \otimes \tau_{d \hspace{0.02in} g_1}(g_2)) \; (\tau_d(g_1) \otimes \id_{A_{g_2}})\\[.1pc]
&\; = \; 
m^A_{g_1,g_2} \; (\id_{A_{g_1}} \otimes \phi_{d \hspace{0.02in} g_1}(g_2) \; \phi_e(g_2)^{-1}) \; (\phi_d(g_1)  \; \phi_e(g_1)^{-1} \otimes \id_{A_{g_2}})\\[-.1pc]
&\; \overset{\eqref{eq:phi}}{=} \; 
 \phi_d(g_1g_2) \;
m^B_{g_1,g_2}\; (\phi_e(g_1)^{-1} \otimes \phi_e(g_2)^{-1})\\[.1pc]
&\; = \; 
 \phi_d(g_1g_2) \;
m^B_{g_1,g_2}\; (\phi_e(g_1)^{-1} \otimes \phi_{g_1}(g_2)^{-1}\; \phi_{g_1}(g_2)\; \phi_e(g_2)^{-1})\\[.1pc]
 &\; \overset{\eqref{eq:phi}(e,g_1,g_2)}{=} \; 
 \phi_d(g_1g_2) \;
 \phi_e(g_1g_2)^{-1} \; m^A_{g_1,g_2} \; (\id_{A_{g_1}} \otimes  \phi_{g_1}(g_2) \;
 \phi_e(g_2)^{-1})\\[.1pc]
 &\; = \; 
\tau_d(g_1 g_2) \; m^A_{g_1,g_2} \; (\id_{A_{g_1}} \otimes \tau_{g_1}(g_2)). 
\end{align*}
Now with this twisting system  $\tau$ for $A$, we can construct an algebra isomorphism $\varphi: B \to A^\tau$. Namely, take $\{\varphi_g:= \phi_e(g): B_g \overset{\sim}{\to} A_g\}_{g \in G}$.
Then we are done by the following computations:
\begin{align*}
 \varphi_{gh}  \; m_{g,h}^B
& \; = \;
\phi_e(gh)  \; m_{g,h}^B \; \overset{\eqref{eq:phi}(e,g,h)}{=} \;
m_{g,h}^A \; (\phi_e(g) \otimes \phi_g(h))\\
& \; = \;
m_{g,h}^A \; (\phi_e(g) \otimes \phi_g(h)\; \phi_e(h)^{-1} \; \phi_e(h)) 
\\[.1pc]
&\; = \;
m_{g,h}^A \; (\phi_e(g) \otimes \tau_g(h) \; \phi_e(h))
\; = \;
(m_{g,h}^A)^\tau \; (\varphi_g \otimes \varphi_h)
\end{align*}
and
\[
\varphi_e\;  u_e^B \;= \;\phi_e(e) \; u_e^B \; \overset{\eqref{eq:phi}}{=} \; u_e^A \; = \; \phi_e(e) \; \phi_e(e)^{-1} \; u_e^A  \; = \; \tau_e(e)^{-1} \; u_e^A \;=\; u_e^{A^\tau}.
\]

\vspace{-.2in}

\end{proof}

\begin{remark}
The second condition of \eqref{eq:phi} on unitality was missing in \cite{Zhang1996}, but we incorporate this here and use it throughout.
\end{remark}

\subsection{Twisted modules in monoidal categories}
\label{sec:twistedmod} 

Finally, we twist modules in monoidal categories via twisting systems. 

\begin{definitionproposition}[$M^\tau$] \label{defprop:Mtau}
Take a twist $\tau$ of $A \in \mathsf{Alg}(\mathsf{Gr}(\cC))$ and $(M, \rho) \in \mathsf{GrMod}(\cC)_A$, we obtain that 
\[
(\{M_g\}_g, \quad \{\rho^\tau_{g,h}:= \rho_{g,h}  (\id_{M_g} \otimes \tau_g(h)): M_g \otimes A_h \to M_{gh}\}_{g,h})
\]
is a right $A^\tau$-module in $\mathsf{Gr}(\cC)$. This module is called the \textnormal{twist of $M$ by $\tau$ in $\mathsf{Gr}(\cC)$},  denoted by $M^\tau$.
\end{definitionproposition}

\begin{proof}
We have that $\rho^\tau$ is associative due to the computation below:
\[
\begin{array}{rl}
\rho^\tau_{gh,\ell} \; (\rho^\tau_{g,h} \otimes \id_{A_\ell})  &= \rho_{gh,\ell}\;(\id_{M_{gh}} \otimes \tau_{gh}(\ell))\;(\rho_{g,h} \otimes \id_{A_\ell}) \;(\id_{M_g} \otimes \tau_g(h) \otimes \id_{A_\ell}) \medskip \\
 &= \rho_{gh,\ell}\;(\rho_{g,h} \otimes \id_{A_\ell})\;(\id_{M_{g}} \otimes \id_{A_h} \otimes \tau_{gh}(\ell))\; (\id_{M_g} \otimes \tau_g(h) \otimes \id_{A_\ell}) \medskip \\
  &= \rho_{g,h\ell}\;(\id_{M_g} \otimes m_{h,\ell})\;(\id_{M_{g}} \otimes \id_{A_h} \otimes \tau_{gh}(\ell)) \;(\id_{M_g} \otimes \tau_g(h) \otimes \id_{A_\ell}) \medskip \\
&= \rho_{g,h\ell}\;(\id_{M_{g}} \otimes \tau_{g}(h\ell))\;(\id_{M_g} \otimes m_{h,\ell}) \;(\id_{M_g} \otimes \id_{A_h} \otimes \tau_h(\ell)) \medskip\\
&= \rho^\tau_{g,h\ell} \;(\id_{M_g} \otimes m^\tau_{h,\ell}).
\end{array}
\]
The first and last equations hold by the definition of $\rho^\tau$; the second equation holds by level-exchange; the third equation follows from the associativity of $\rho$; and the fourth equation follows from the twisting condition \eqref{eq:twistcond}$(g,h,\ell)$.

We have that the right unitality axioms hold due to the computation below:
\begin{align*}
\rho^\tau_{g,e} \; (\id_{M_g} \otimes u^\tau_e) 
&\; = \; \rho_{g,e} \;(\id_{M_g} \otimes \tau_g(e))\;(\id_{M_g} \otimes \tau_e(e)^{-1})\;(\id_{M_g} \otimes u_e)\\
&\; \overset{\textnormal{Lem.~\ref{lem:tau-inv}}}{=} \; \rho_{g,e}\;(\id_{M_g} \otimes \tau_g(e))\;(\id_{M_g} \otimes \tau_g(e)^{-1})\;(\id_{M_g} \otimes u_e) \;  \overset{ \rho \textnormal{ unital} }{=} \; \id_{M_g}.
\end{align*}

\vspace{-.25in}

\end{proof}

\begin{example} \label{ex:twistmods}
Similar to Example~\ref{ex:twistalgs}, we can twist graded modules in $\cC$ via an identity twist, an inverse twist, or a composite twist. Namely, we have the following examples.
\begin{enumerate}[\upshape (a)]
\item For the twist $\tau^{\textnormal{id}}$ of $A$, with $M \in \mathsf{GrMod}(\cC)_A$, we get that $M^{\tau^{\textnormal{id}}} = M$ in $\mathsf{GrMod}(\cC)_A$.

\smallskip

\item Given a twist $\tau$ of $A$, with $M \in \mathsf{GrMod}(\cC)_A$, we get that $(M^{\tau})^{\tau^{-1}} = M$ in $\mathsf{GrMod}(\cC)_A$.
\smallskip

\item Given a twist $\tau$ of $A$ and a twist $\sigma$ of $A^\tau$, we obtain that $(M^{\tau})^{\sigma} = M^{\tau \sigma}$ in $\mathsf{GrMod}(\cC)_{A^{\tau \sigma}}$.
\end{enumerate}
\end{example}

%%%%%%%%%%%%%%%%%%%%%%%%%%
%%%%%%%%%%%%%%%%%%%%%%%%%%
%%%%%%%%%%%%%%%%%%%%%%%%%%
\section{Zhang-Morita equivalence in monoidal categories} \label{sec:mainresults}

The aim of this section is to generalize the main results in \cite[Section~3]{Zhang1996} on twisted algebras and Zhang-Morita equivalence to the monoidal setting. 
We show how a twist equivalence of graded algebras in the monoidal setting yields Zhang-Morita (ZM-)equivalence in Section~\ref{sec:forward}. Then we provide a version of a converse statement to this result, in the enriched setting, in Section~\ref{sec:converse}.

\subsection{Twist equivalence yields ZM-equivalence} \label{sec:forward}
Here, we generalize \cite[Theorem~3.1]{Zhang1996} to the monoidal setting.

\begin{theorem} \label{thm:mainfwd}
Recall Definition-Propositions~\ref{defprop:Atau} and~\ref{defprop:Mtau}. Take algebras $A, B$ in $\mathsf{Gr}(\cC)$, with a twisting system $\tau$ for $A$. If $A^\tau \cong B$ as algebras in $\mathsf{Gr}(\cC)$, then we have an isomorphism of categories: $$\mathsf{GrMod}(\cC)_A\; \cong \;\mathsf{GrMod}(\cC)_{B}.$$  In particular,  $A$ and $B$ are ZM-equivalent  when they are twist equivalent as algebras in $\mathsf{Gr}(\cC)$.
\end{theorem}

\begin{proof}
Given that $A^\tau \cong B$, for some twisting system $\tau$ for $A$,  we define a functor:
\[
F_1: \mathsf{GrMod}(\cC)_A \longrightarrow \mathsf{GrMod}(\cC)_{A^\tau},
\quad M \mapsto M^\tau, \quad (\varphi:  M \to N) \mapsto (\varphi:  M^\tau \to N^\tau).
\]
Indeed, $F_1$ is well-defined on objects due to Definition-Proposition~\ref{defprop:Mtau}. Also, $F_1$ is well-defined on morphisms, since for a morphism  $\varphi: M \to N$ in  $\mathsf{GrMod}(\cC)_A$, and for each $g,h \in G$, we get that:
\begin{align*}
\rho_{g,h}^{\tau}(\varphi_g \otimes \id_{A_h}) &\; = \;
\rho_{g,h}(\id_{M_g} \otimes \tau_g(h))(\varphi_g \otimes \id_{A_h})\\[.2pc]
 &\; = \;
\rho_{g,h}(\varphi_g \otimes \id_{A_h})(\id_{M_g} \otimes \tau_g(h))\\[.2pc]
&\; = \; 
\varphi_{gh} \;\rho_{g,h} \; (\id_{M_g} \otimes \tau_g(h))\\[.2pc]
&\; = \; 
\varphi_{gh}\; \rho_{g,h}^{\tau}.
\end{align*}
Here, the second equation holds by level-exchange, and the third equation holds as $\varphi \in \mathsf{GrMod}(\cC)_A$. Thus, $\varphi:  M^\tau \to N^\tau$ belongs to $\mathsf{GrMod}(\cC)_{A^\tau}$. 

For the other direction, define the functor below in similar manner to $F_1$ above, using the inverse twisting system $\tau^{-1}$ on $A^\tau$ from Example~\ref{ex:twistalgs}(b):
\[
F_{-1}: \mathsf{Gr}(\cC)_{A^\tau} \longrightarrow \mathsf{GrMod}(\cC)_A,
\quad M \mapsto M^{\tau^{-1}}, \quad (\varphi:  M \to N) \mapsto (\varphi:  M^{\tau^{-1}} \to N^{\tau^{-1}}).
\]
Indeed, this is well-defined as $(A^{\tau})^{\tau^{-1}} = A$ as algebras in $\mathsf{Gr}(\cC)$; see Example~\ref{ex:twistalgs}. 

Finally, $F_1$ and $F_{-1}$ are mutually inverse functors due to Example~\ref{ex:twistmods}. Namely, if $M$ is in $\mathsf{GrMod}(\cC)_A$, then $(M^{\tau})^{\tau^{-1}} = M$ as right $A$-modules in $\mathsf{Gr}(\cC)$. Likewise,  if $M \in \mathsf{GrMod}(\cC)_{A^\tau}$, then $(M^{\tau^{-1}})^{\tau} = M$  in $\mathsf{GrMod}(\cC)_{A^\tau}$.

For the last statement, note that the isomorphism of categories, $\mathsf{GrMod}(\cC)_A \cong \mathsf{GrMod}(\cC)_{B}$, induces a $\cC$-enriched structure on $\mathsf{GrMod}(\cC)_{B}$ from that on $\mathsf{GrMod}(\cC)_{A}$ [Corollary~\ref{cor:GrModCA-main}]. With this induced structure, we obtain that $\mathsf{GrMod}(\cC)_A \simeq^\cC \mathsf{GrMod}(\cC)_{B}$.
\end{proof}

%%%%%%%%%%%%%%%%%%%%%%%%%%
\subsection{How ZM-equivalence yields twist equivalence} 
\label{sec:converse}
Our final goal is to produce a type of converse result of Theorem~\ref{thm:mainfwd}. In particular, we will generalize \cite[Theorem~3.3]{Zhang1996} to the monoidal setting. We begin by introducing some algebras in the enriched setting that are analogous to the algebras $\Gamma(A)$ in \cite[Section~3]{Zhang1996}.

\begin{definitionproposition}[$\Gamma^{\cV}(X)$] \label{def:gamma} 
 Let $\cD$ be a $G$-graded $\cV$-category, and take $X \in \cD$. Then we can form an algebra $\Gamma^{\cV}(X)$ in $\mathsf{Gr}(\cV)$ as follows.
\begin{enumerate}[\upshape(a)]
\item $\Gamma^{\cV}(X) := \big(\cD^{\cV}(S_g(X),X)\big)_{g \in G}$ as an object in $\mathsf{Gr}(\cV)$.
\skp
\item The multiplication maps are given by the formula below, for $g,h \in G$:  {\small
\[
\qquad m_{g,h}^{\Gamma^{\cV}(X)}:=
\gamma^{\cV}_{S_{gh}(X),S_g(X),X} \left(\id \otimes^{\cV} (S_g)_{S_h(X),X}\right): \cD^{\cV}(S_g(X),X) \otimes^{\cV} \cD^{\cV}(S_h(X),X) \to \cD^{\cV}(S_{gh}(X),X).
\]
}

\item The unit map is given by: 
\[
u_e^{\Gamma^{\cV}(X)}:= \kappa^{\cV}_X: \one^{\cV} \to \cD^{\cV}(X,X). 
\]
\end{enumerate}
We refer to these graded algebras as \textnormal{gamma algebras}.
\end{definitionproposition}

\begin{proof}
It suffices to show that the multiplication map is associative, and unital with respect to the unit map. Now associativity holds by the following computation, for $g,h,\ell \in G$:

\vspace{-.1in}

\[
{\footnotesize
\begin{array}{l}
m^{\Gamma^{\cV}(X)}_{gh, \ell}\left(m^{\Gamma^{\cV}(X)}_{g, h} \otimes^{\cV} \id_{(\Gamma^{\cV}(X))_{\ell}}\right)
\; =
m^{\Gamma^{\cV}(X)}_{gh, \ell}\left(\gamma^{\cV}_{S_{gh}(X),S_g(X),X}(\id \otimes^{\cV} (S_g)_{S_h(X),X}) \otimes^{\cV}\id_{(\Gamma^{\cV}(X))_{\ell}}\right)\\[.6pc]
=
\gamma^{\cV}_{S_{gh\ell}(X),S_{gh}(X),X}
\left(\id \otimes^{\cV} (S_{gh})_{S_{\ell}(X),X}\right)
\left(\gamma^{\cV}_{S_{gh}(X),S_g(X),X}(\id \otimes^{\cV} (S_g)_{S_h(X),X}) \otimes^{\cV}\id_{(\Gamma^{\cV}(X))_{\ell}}\right)\\[.6pc]
=
\gamma^{\cV}_{S_{gh\ell}(X),S_{gh}(X),X}
\left(\gamma^{\cV}_{S_{gh}(X),S_{g}(X),X} \otimes^{\cV} \id \right)
\left(\id \otimes^{\cV} (S_g)_{S_h(X),X}  \otimes^{\cV} (S_g)_{S_{h\ell}(X),S_h(X)} \right)
\left(\id \otimes^{\cV} \id \otimes^{\cV}(S_h)_{S_\ell(X),X} \right)\\[.6pc]
=
\gamma^{\cV}_{S_{gh\ell}(X),S_{g}(X),X}
\left(\id \otimes^{\cV}\gamma^{\cV}_{S_{gh\ell}(X),S_{gh}(X),S_g(X)} \right)
\left(\id \otimes^{\cV} (S_g)_{S_h(X),X}  \otimes^{\cV} (S_g)_{S_{h\ell}(X),S_h(X)} \right)
\left(\id \otimes^{\cV} \id \otimes^{\cV}(S_h)_{S_\ell(X),X} \right)\\[.6pc]
=
\gamma^{\cV}_{S_{gh\ell}(X),S_{g}(X),X}
\left(\id \otimes^{\cV} (S_g)_{S_{h\ell}(X),X} \right)
\left(\id \otimes^{\cV}\gamma^{\cV}_{S_{h\ell}(X),S_{h}(X),X} \right)
\left(\id \otimes^{\cV} \id \otimes^{\cV}(S_h)_{S_\ell(X),X} \right)\\[.6pc]
= m^{\Gamma^{\cV}(X)}_{gh, \ell}
\left( \id_{(\Gamma^{\cV}(X))_{g}}
\otimes^{\cV}
m^{\Gamma^{\cV}(X)}_{h, \ell} \right).
\end{array}
}
\]

\vspace{.1in}

\noindent The first two equations and last equation hold by definition. The third equation holds by level-exchange. The fourth equation holds by the associativity of $\gamma^{\cV}$. The fifth equation holds since $S_g$ is a $\cC$-functor and is compatible with $\gamma^{\cV}$.

To verify the left unitality axiom, we have:
\[
{\small
\begin{array}{rl}
m^{\Gamma^{\cV}(X)}_{e, g}\left(u_e^{\Gamma^{\cV}(X)} \otimes^{\cV} \id_{\Gamma(X)_g} \right)
&= \; 
\gamma^{\cV}_{S_{g}(X),X,X} \left(\id_{\Gamma(X)_e} \otimes^{\cV} (S_e)_{S_g(X),X}\right)\left(\kappa^{\cV}_X \otimes^{\cV} \id_{\Gamma(X)_g} \right)\\[.6pc]
& = \; 
\gamma^{\cV}_{S_{g}(X),X,X}\left(\kappa^{\cV}_X \otimes^{\cV} \id_{\Gamma(X)_g} \right) \;=\; \id_{\Gamma(X)_g}.
\end{array}
}
\]
The first two equations hold by definition, and the last equation is the left unitality of a $\cV$-category.

To verify the right unitality axiom, we have:
\[
{\small
\begin{array}{rl}
m^{\Gamma^{\cV}(X)}_{g, e}\left(\id_{\Gamma(X)_g} \otimes^{\cV} u_e^{\Gamma^{\cV}(X)} \right)
&= \; 
\gamma^{\cV}_{S_{g}(X),S_{g}(X),X} \left(\id_{\Gamma(X)_g} \otimes^{\cV} (S_g)_{X,X}\right)\left(\id_{\Gamma(X)_g} \otimes^{\cV} \kappa^{\cV}_X \right)\\[.6pc]
&= \; 
\gamma^{\cV}_{S_{g}(X),S_{g}(X),X}
\left(\id_{\Gamma(X)_g} \otimes^{\cV} \kappa^{\cV}_{S_g(X)} \right).
\end{array}
}
\]
The first  equation holds by definition. The second equation holds since $S_g$ is a $\cV$-functor and is compatible with $\kappa^{\cV}$.
The last equation is the right unitality of a $\cV$-category.
\end{proof}

Consider the following special case.

\begin{example}[$\Gamma^{\cC}(A)$] \label{ex:GammaCA}

Consider $\mathsf{GrMod}(\cC)_A$ as a $G$-graded $\cC$-category as in Proposition~\ref{prop:GrModCA-gr}. Then $\Gamma^\cC(A)=[\![A,A]\!]_A$ as graded algebras in $\cC$, for the regular module $A \in\mathsf{GrMod}(\cC)_A$. Indeed:
\[\Gamma^{\cC}(A) := \big(\langle[\![S_g(A),A]\!]_A\rangle_e\big)_{g \in G}= \left(\left([\![A,A]\!]_A\right)_g\right)_{g\in G}=:[\![A,A]\!]_A.\] 
as an object in $\mathsf{Gr}(\cC)$, where the second equality follows from Lemma~\ref{claim:S_g-props}(b). 

The unit map of $\Gamma^{\cC}(A)$ is given by $(\kappa^A_A)_e:\one\rightarrow\Gamma^{\cC}(A)_e:=\langle[\![A,A]\!]_A\rangle_e$, which corresponds to the map $\kappa^A_A:\one\rightarrow[\![A,A]\!]_A$ under our isomorphism $\mathsf{GrAlg}(\cC)\cong\mathsf{Alg}(\mathsf{Gr}(\cC))$. Similarly, using Lemma~\ref{claim:S_g-props} and the isomorphism $\mathsf{GrAlg}(\cC)\cong\mathsf{Alg}(\mathsf{Gr}(\cC))$, the multiplication map of $\Gamma^{\cC}(A)$ is equal to:

\vspace{-.1in}
    \[
    \gamma^A_{A,A,A}:[\![A,A]\!]_A \; \bar{\otimes} \; [\![A,A]\!]_A\rightarrow[\![A,A]\!]_A
    \]
    
\vspace{.05in}

\noindent Hence the structure maps of $\Gamma^{\cC}(A)$ as a graded algebra coincide with those of $[\![A,A]\!]_A$.
\end{example}
In light of the example above, we get the following:

\begin{corollary} \label{cor:Gamma-A}
Given $A \in \mathsf{Alg}(\mathsf{Gr}(\cC))$, we obtain that  $\Gamma^{\cC}(A) \cong A$ as algebras in  $\mathsf{Gr}(\cC)$. 
\end{corollary}
\begin{proof}
    By the example above, $\Gamma^\cC(A)=[\![A,A]\!]_A$. Now by Theorem~\ref{thm:endomalg}, $\Gamma^\cC(A)=[\![A,A]\!]_A\cong A$.
\end{proof}

Now we turn our attention towards obtaining a twist equivalence of gamma algebras.

\begin{theorem} \label{thm:Gamma-twist}
 Take $G$-graded $\cV$-categories $\cD$ and $\cE$, and fix objects $X \in \cD$ and $Y \in \cE$. Suppose that there exists a $\cV$-fully faithful functor $\Phi: \cD \to \cE$, such that $\Phi(S_g(X)) \cong S_g(Y)$ as objects in $\cE_0$, for each $g \in G$. Then,  we have that
\[
\left(\Gamma^{\cV}(X)\right)^T \; \cong \; 
\Gamma^{\cV}(Y)
\]
as algebras in $\mathsf{Gr}(\cV)$, for some twisting system $T$ on $\Gamma^{\cV}(X)$.
\end{theorem}

\begin{proof}
 By Proposition~\ref{prop:phi}, it suffices to define isomorphisms in $\cV$,
\[
\{\phi_d(g): \cE^\cV(S_g(Y),Y) \overset{\sim}{\longrightarrow}  \cD^\cV(S_g(X),X)\}_{d,g \in G},
\]
such that \eqref{eq:phi} holds, that is, the equations below need to hold:
\begin{equation} \label{eq:phi-app}
m_{g_1,g_2}^{\Gamma^{\cV}_{\cD}(X)} \circ \left(\phi_d(g_1) \otimes^{\cV} \phi_{dg_1}(g_2)\right) \; = \;
\phi_d(g_1 g_2) \circ m_{g_1,g_2}^{\Gamma^{\cV}_{\cE}(Y)},
\end{equation}

\vspace{-.2in}

\begin{equation} \label{eq:phi-app-2}
\phi_e(e) \circ u_e^{\Gamma^{\cV}_{\cE}(Y)} \; = \;
u_e^{\Gamma^{\cV}_{\cD}(X)},
\end{equation}
for all $d, g_1, g_2 \in G$. To proceed, we need to define a few morphisms.
\skp

\begin{itemize}
\item Take $t_g: \Phi(S_g(X)) \overset{\sim}{\to} S_g(Y)$ to be the given isomorphism in $\cE_0$, for $g \in G$.
\skp
\item For $f \in \Hom_{\cE_0}(Z,Z') \cong \Hom_{\cV}(\one^{\cV}, \cE^{\cV}(Z,Z'))$ and $W \in \cE$, define:
\[
f_*^W:=\gamma^{\cV}_{W,Z,Z'}\left(f \otimes^{\cV} \id_{\cE^{\cV}(W,Z)}\right): \; \cE^{\cV}(W,Z) \to \cE^{\cV}(W,Z').
\]
\item For $g \in \Hom_{\cE_0}(Z',Z) \cong \Hom_{\cV}(\one^{\cV}, \cE^{\cV}(Z',Z))$ and $W \in \cE$, define:
\[
g^*_W:=\gamma^{\cV}_{Z',Z,W}\left(\id_{\cE^{\cV}(Z,W)} \otimes^{\cV} g\right): \; \cE^{\cV}(Z,W) \to \cE^{\cV}(Z',W).
\]
\end{itemize}

\begin{claim} \label{claim5}
 The morphisms $\{\phi_d(g):\cE^\cV(S_g(Y),Y) \overset{\sim}{\longrightarrow}  \cD^\cV(S_g(X),X)\}_{d,g \in G}$ defined by 
\[
\phi_d(g):=
\left((S_{d^{-1}})_{S_{dg}(X),S_d(X)}\right) \circ
\left(\Phi_{S_{dg}(X),S_d(X)}\right)^{-1} \circ
(t_d^{-1})_*^{\Phi(S_{dg}(X))} \circ
(t_{dg})^*_{S_d(Y)} \circ
\left((S_d)_{S_{g}(Y),Y}\right)
\]
satisfies the conditions~\eqref{eq:phi-app} and~\eqref{eq:phi-app-2}.
\end{claim}
We reserve the proof of the claim for Appendix~\ref{app:Sec5.2claims}. This concludes the proof of the theorem.
\end{proof}

Finally, this brings us to our enriched converse result of Theorem~\ref{thm:mainfwd}.

\begin{corollary} \label{cor:backward}
Take algebras $A$ and $B$ in  $\mathsf{Gr}(\cC)$, and recall the $G$-graded, $\cC$-enriched structure on the categories $\mathsf{GrMod}(\cC)_A$ and $\mathsf{GrMod}(\cC)_B$ from Corollary~\ref{cor:GrModCA-main} and Proposition~\ref{prop:GrModCA-gr}.  Suppose that $A$ and $B$ are ZM-equivalent, namely \[\mathsf{GrMod}(\cC)_A \; \simeq^{\cC} \;  \mathsf{GrMod}(\cC)_B,\]  via a $\cC$-functor $\Phi$. Moreover, assume that $\Phi(S_g(A)) \cong S_g(B)$ as objects in $(\mathsf{GrMod}(\cC)_B)_0$, for all $g \in G$. Then $A^\tau \cong B$ as algebras in $\mathsf{Gr}(\cC)$, for some twisting system $\tau$ on $A$. 
\end{corollary}

\begin{proof}
Note that the $\cC$-equivalence $\Phi$ here is $\cC$-fully faithful by definition. So, the hypotheses of this result fulfill the hypotheses of Theorem~\ref{thm:Gamma-twist} with $\cV := \cC$, $\cD := \mathsf{GrMod}(\cC)_A$, $\cE := \mathsf{GrMod}(\cC)_B$, $X := A$, and $Y:= B$. Thus, $\Gamma^{\cC}(A)^T \cong \Gamma^{\cC}(B)$  in $\mathsf{Alg}(\mathsf{Gr}(\cC))$, for some twisting system $T$ on $\Gamma^{\cC}(A)$. By  Corollary~\ref{cor:Gamma-A}, there exists an isomorphism $\psi: \Gamma^{\cC}(A) \overset{\sim}{\to} A$ in $\mathsf{Alg}(\mathsf{Gr}(\cC))$. Therefore, $A^\tau \cong B$ as algebras in  $\mathsf{Gr}(\cC)$, where $\tau:= \psi(T)$ is a twisting system on $A$.
\end{proof}

%%%%%%%%%%%%%%%%%%%%%%%%%%
%%%%%%%%%%%%%%%%%%%%%%%%%%
%%%%%%%%%%%%%%%%%%%%%%%%%%

\appendix

\section{Proofs of selected results}

Here, we provide several proofs of results above. In Section~A.1, we prove the details of the proof of Proposition~\ref{prop:alg-bar}. Next, in  Section~A.2 (resp., in Section~A.3), we prove Claims~\ref{claim1} and~\ref{claim2} (resp., Claims~\ref{claim3} and~\ref{claim4}) for Section~\ref{sec:graded-closed-enrich} results. Finally, in Section~A.4, we verify Claim~\ref{claim5} for a result in Section~\ref{sec:converse}.

\subsection{Proof of Proposition~\ref{prop:alg-bar}} 
\label{sec:prop:alg-bar}
Suppose that $(A, m, u)$ is an algebra in $\mathsf{Gr}(\cC)$.    
After taking degrees, the data of the multiplication map $m:A \; \bar{\otimes}\; A \rightarrow A$ amounts to a collection of morphisms:
    \[ \{ m_q: \textstyle \bigoplus_{\ell\in G} A_\ell\otimes A_{\ell^{-1}q} \rightarrow A_q \}_{q \in G}\]
in $\cC$. Moreover, consider the canonical inclusions: 
    \[ \{ \iota_{\ell,q}: A_\ell\otimes A_{\ell^{-1}q}\hookrightarrow \textstyle \bigoplus_{p\in G} A_p\otimes A_{p^{-1}q} \}_{\ell, q \in G}.\]
We now obtain a candidate for the collection of multiplication morphisms:
    \[\{ m_{g,h}:= m_{gh} \circ\iota_{g,gh} : A_g\otimes A_{h} \rightarrow A_{gh}\}_{g,h \in G} .\]
The associativity constraints for these morphisms are verified via the commutative diagram below.
\[
{\footnotesize
\xymatrix@C=-1.5pc@R=.7pc{
A_g \otimes A_{g^{-1}h} \otimes A_{h^{-1}k}
\ar[rrrr]^{m_{g,g^{{\text -}1}h} \; \otimes \;  \id}
\ar[rrd]^{\iota_{g,h} \; \otimes \; \id}
\ar[rddd]^{\id \; \otimes \; \iota_{h, g^{{\text -}1} k}}
\ar[dddddd]_{\id \; \otimes \; m_{g^{{\text -}1}h, h^{{\text -}1}k}}
&&&& A_h \otimes A_{h^{-1}k}
\ar[dddddd]^{m_{h,h^{{\text -}1}k}}
\ar[lddd]_{\iota_{h,k}}\\
&& \left(\bigoplus_p A_p \otimes A_{p^{-1}h}\right) \otimes A_{h^{-1}k} 
\ar[rru]^{m_h \; \otimes \; \id}
\ar[d]^(.5){\iota}&&\\
&& \bigoplus_q \left(\bigoplus_p A_p \otimes A_{p^{-1}q}\right) \otimes A_{q^{-1}k} 
\ar[rd]_{\bigoplus_q m_q \; \otimes \; \id\; \; \;}
\ar[dd]^(.6){\cong}&&\\
& A_g \otimes \left(\bigoplus_q A_{g^{-1}q} \otimes A_{q^{-1}k}\right) 
\ar[rd]^(.55){\iota} 
\ar[lddd]^{\id \; \otimes \; m_{g^{{\text -}1}k}}
&& \bigoplus_q A_q \otimes A_{q^{-1}k} 
\ar[rddd]_{m_k}&\\
&& \bigoplus_p A_p \otimes \left(\bigoplus_q A_{p^{-1}q} \otimes A_{q^{-1}k}\right)
\ar[d]^(.55){\bigoplus_p \id \; \otimes \; m_{p^{{\text -}1}k}}&&\\
&& \bigoplus_p A_p \otimes A_{p^{-1}k} 
\ar[rrd]^{m_k}&&\\
A_g \otimes A_{g^{-1}k} 
\ar[rru]^{\iota_{g,k}} 
\ar[rrrr]_{m_{g,g^{{\text -}1}k}} &&&& A_k
}
}
\]
Here, the outer triangles hold by definition; the top left region commutes by associativity of coproducts; 
%[\fll{see proof here: https://math.stackexchange.com/questions/1081690/associativity-of-coproducts }] 
the bottom right region commutes by the associativity of $m$; and the top right and bottom left regions commute by the universal property of the coproduct as shown in \eqref{eq:coproduct} below.
\begin{equation}
\label{eq:coproduct}
{\small
\begin{tikzcd}
	{X_g} & {\bigoplus_p X_p} \\
	{Y_g} & {\bigoplus_p Y_p}
	\arrow["\iota_g^X", hook, from=1-1, to=1-2]
	\arrow["{\varphi_g}"', from=1-1, to=2-1]
	\arrow["\iota_g^Y"', hook, from=2-1, to=2-2]
	\arrow["{\bigoplus_p \varphi_p}", from=1-2, to=2-2]
\end{tikzcd}
}
\end{equation}

 Since $\bar{\one}_g=\delta_{g,e}\one$, the unit morphism $u$ is determined by the map $u_e: \one\rightarrow A_e$. The left unitality axiom holds by the commutative diagram below, and the right unitality axiom holds likewise.
 \[
{\small
\xymatrix@C=9pc@R=1.5pc{
\one \otimes A_g 
\ar[r]^{u_e \; \otimes \; \id}
\ar[d]_{\cong}
& A_e \otimes A_g
\ar[d]^{\iota_{e,g}}
\ar@/^4pc/[dd]^{m_{e,g}} \\
\bigoplus_p \one_p \otimes A_{p^{-1}g} 
\ar[r]^{(u_e \; \oplus  \; \bigoplus_{p \neq e} \vec{0}) \;  \otimes \; \id}
\ar[rd]_{\cong}
& \bigoplus_p A_p \otimes A_{p^{-1}g}
\ar[d]^{m_g}\\
& A_g
}
}
 \]
The top region commutes by the universal property of coproduct as in \eqref{eq:coproduct}; the right region commutes by definition; and the bottom triangle commutes by the unit axiom for $u$. Thus, we constructed a $G$-graded algebra in $\cC$, e.g., an object of $\mathsf{GrAlg}(\cC)$.

Conversely, given a $G$-graded algebra $(\{A_g\}_{g \in G},\;\{m_{g,h}\}_{g,h\in G},\; u_e)$ in $\cC$, take $A$ to be the tuple of objects $(A_g)_{g \in G}$. Moreover, let $m : A\; \bar{\otimes}\; A \to A$ be written as $(m_g: \bigoplus_{h} A_h  \otimes A_{h^{-1}g} \to A_g)_{g \in G}$, where $m_g$ is the biproduct $\bigoplus_h m_{h,h^{-1}g}$ of multiplication morphisms. Lastly, let $u: \bar{\one} \to A$ be written as $(u_g: \one_g \to A_g)_{g \in G}$, where $u_e: \one \to A_e$ is the given unit morphism, and $u_g: 0 \to A_g$ is the zero morphism for $g \neq e$. It is straight-forward to check that the associativity and unitality constraints of the morphisms of the $G$-graded algebra then imply the associativity and unitality axioms for $m$ and $u$. Thus, we constructed an algebra in $\mathsf{Gr}(\cC)$.

Moreover, the bijective correspondence between objects of $\mathsf{Alg}(\mathsf{Gr}(\cC))$ and $\mathsf{GrAlg}(\cC)$ above also extends to the desired  bijective correspondence between morphisms. 
\qed

\subsection{Proofs of Claims~\ref{claim1} and~\ref{claim2}} \label{app:Sec3.3claims}

%%%%%%%%%%%%%%%%%%%%%%%%%%
%%%%%%%%%%%%%%%%%%%%%%%%%%
%%%%%%%%%%%%%%%%%%%%%%%%%%

\begin{proof}[Proof of Claim~\ref{claim1}] 
(a) By the universal property of $\textnormal{eq}: [M,P]_A \to [M,P]$, it suffices to show that $\gamma^\cC_{M,N,P}\circ (\textnormal{eq} \otimes \textnormal{eq})$ equalizes $R_{M,P}$ and $S_{M,P}$. By Lemma~\ref{lem:equalizers}, we will instead show that:
\[ 
\xymatrix@C=4pc{
[N,P]_A\otimes[M,N]_A\otimes M\otimes A
\ar[r]^(.52){\textnormal{eq}\otimes\textnormal{eq}\otimes \id \otimes \id}
&[N,P]\otimes[M,N]\otimes M\otimes A
\ar[r]^(.56){\gamma^\cC \otimes \id \otimes \id}
&[M,P]\otimes M\otimes A
}
\]
equalizes $R_{M,P}^\#$ and $S_{M,P}^\#$. This is showcased in Figure~\ref{fig:C_A-comp}. Here top and bottom-center-right hexagons commute by \eqref{eq:equal-sharp}; the bottom-center-left and bottom-right rectangular regions commute by $(\gamma^\cC)^\# = ((\epsilon(\id \otimes \epsilon))^\flat)^\# = \epsilon(\id \otimes \epsilon)$ [\eqref{eq:sharp-flat}, \eqref{eq:right-closed}, Proposition~\ref{prop:self-enriched}(iv)]; and the remaining rectangular regions commute by level-exchange.
\begin{figure}[h!]
 \[
 \hspace{-.1in}
{\footnotesize
\xymatrix@C=-2pc@R=1.3pc{
% ROW 1
 & & [N,P]_A\otimes[M,N]_A\otimes M\otimes A 
    \ar[dll]_{\id \; \otimes \; \textnormal{eq} \; \otimes \; \id \; \otimes \; \id}
    \ar[drrr]^{\id \; \otimes \; \textnormal{eq} \; \otimes \; \id \; \otimes \; \id}
 & & & \\
% ROW 2
[N,P]_A\otimes[M,N]\otimes M\otimes A 
    \ar[dd]_{\textnormal{eq} \; \otimes \; \id \; \otimes \; \id \; \otimes \; \id}
    \ar[dr]^(.6){\id \; \otimes \; \id \; \otimes \; \rho}
& & & & & [N,P]_A\otimes[M,N]\otimes M\otimes A
    \ar[dd]^{\textnormal{eq} \; \otimes \; \id \; \otimes \; \id \; \otimes \; \id}
    \ar[dl]_(.6){\id \; \otimes \; \epsilon \; \otimes \; \id}
\\
% ROW 3
 & [N,P]_A\otimes[M,N]\otimes M 
    \ar[dd]^{\textnormal{eq} \; \otimes \; \id \; \otimes \; \id }
    \ar[dr]^{\id \; \otimes \; \epsilon}
 & &  & [N,P]_A\otimes N\otimes A
    \ar[dd]^(.3){\textnormal{eq} \; \otimes \; \id \; \otimes \; \id }
    \ar[dll]_{\id \; \otimes \; \rho}
    \ar[dddl]_{\textnormal{eq} \; \otimes \; \id \; \otimes \; \id }
 & \\
% ROW 4
[N,P]\otimes[M,N]\otimes M\otimes A
    \ar[dd]_{\gamma^\cC \; \otimes \; \id \; \otimes \; \id }
    \ar[dr]_(.4){\id \; \otimes \; \id \; \otimes \; \rho }
& & [N,P]_A\otimes N
    \ar[ddd]_{\textnormal{eq} \; \otimes \; \id}
& & & [N,P]\otimes[M,N]\otimes M\otimes A
    \ar[dd]^{\gamma^\cC \; \otimes \; \id \; \otimes \; \id }
    \ar[dl]^(.45){\id \; \otimes \; \epsilon \; \otimes \; \id }
\\
% ROW 5
 & [N,P]\otimes[M,N]\otimes M
    \ar[dd]_{\gamma^\cC \; \otimes \; \id }
    \ar[ddr]_{\id \; \otimes \; \epsilon}
 & & & [N,P]\otimes N\otimes A
    \ar[dd]^(.4){\epsilon \; \otimes \; \id }
 & \\
% ROW 6
[M,P]\otimes M\otimes A
    \ar[dr]_{\id \; \otimes \; \rho }
& & & \; \; \; [N,P]\otimes N\otimes A
    \ar[dl]^{\id \; \otimes \; \rho }
& & [M,P]\otimes M\otimes A
    \ar[dl]^{\epsilon \; \otimes \; \id }
\\
% ROW 7
 & [M,P]\otimes M
    \ar[dr]_{\epsilon} 
 & [N,P]\otimes N
    \ar[d]_{\epsilon}
 & & P\otimes A
    \ar[dll]^{\rho}
& \\
% ROW 8
 & & P & & & \\
}
}
 \]
\caption{Proof of Claim~\ref{claim1}(a): Defining the composition morphisms $\gamma^A$.}
\label{fig:C_A-comp}

\end{figure}
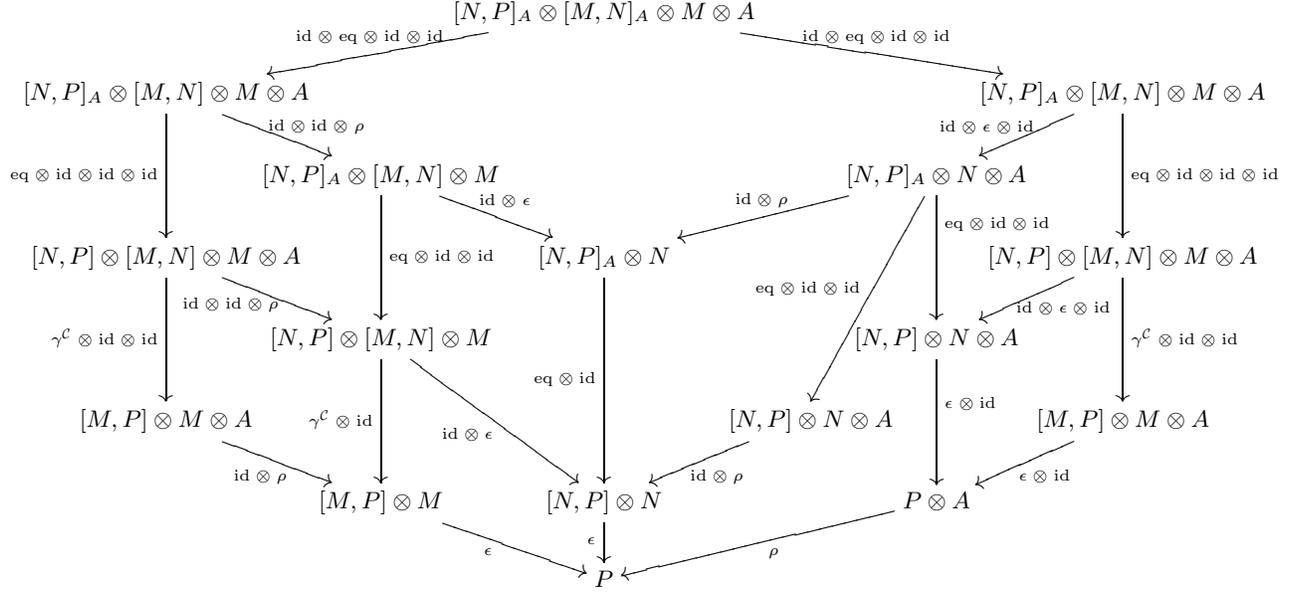

\skp 

(b) Associativity of composition is showcased in Figure~\ref{fig:C_A-assoc}. 
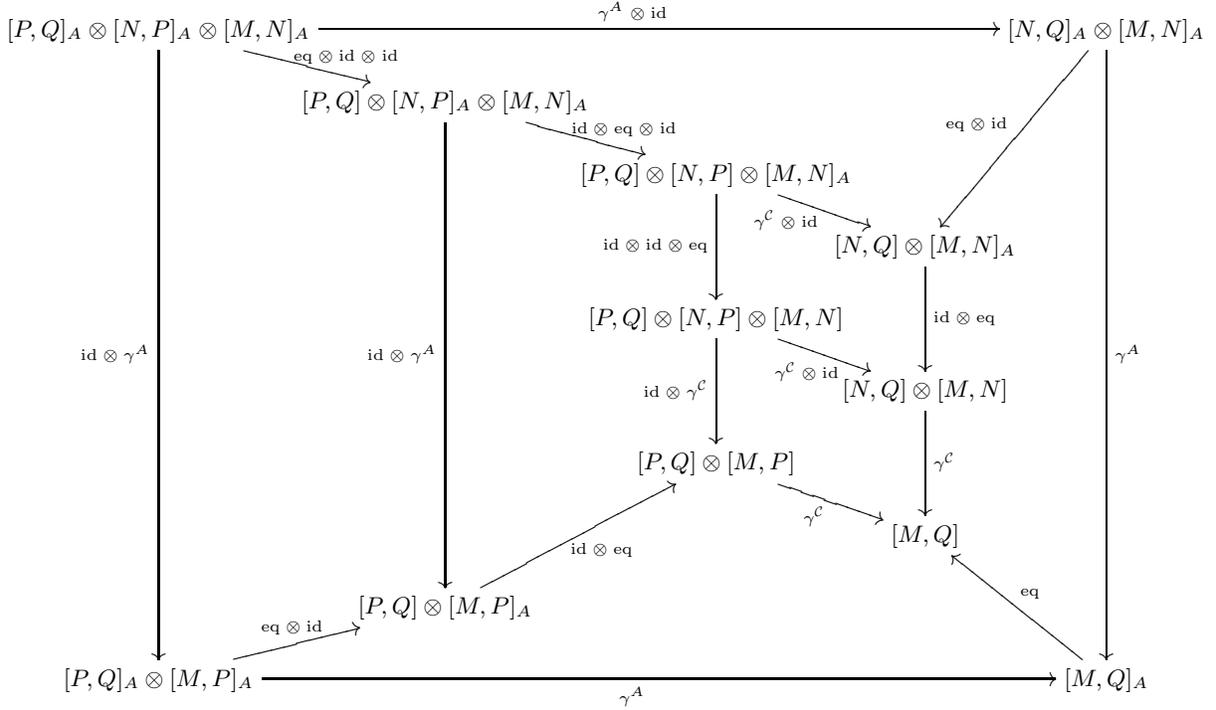
\begin{figure}[h!]
 \[
{\footnotesize
\xymatrix@C=-0.8pc@R=1pc{
[P,Q]_A\otimes[N,P]_A\otimes[M,N]_A
\ar[rrrrr]^{\gamma^A \; \otimes \; \id}
\ar[ddddddddd]_{\id \; \otimes \; \gamma^A}
\ar[rd]^(.6){\textnormal{eq} \; \otimes \; \id \; \otimes \; \id}
&&&&&
[N,Q]_A\otimes[M,N]_A
\ar[ddddddddd]^{\gamma^A}
\ar[lddd]_{\textnormal{eq} \; \otimes \; \id}
\\
&
[P,Q]\otimes[N,P]_A\otimes[M,N]_A
\ar[rd]^(.6){\id \; \otimes \; \textnormal{eq} \; \otimes \; \id}
\ar[ddddddd]_{\id \; \otimes \; \gamma^A}
&&&&\\
&&
[P,Q]\otimes[N,P] \otimes[M,N]_A
\ar[rrd]_(.4){\gamma^\cC \; \otimes \; \id}
\ar[dd]_{\id \; \otimes \; \id \; \otimes \; \textnormal{eq}}
&&&\\
&&&&
[N,Q]\otimes[M,N]_A
\ar[dd]^{\id \; \otimes \; \textnormal{eq}}
&
\\
&&
[P,Q]\otimes[N,P]\otimes[M,N]
\ar[rrd]_(.5){\gamma^\cC \; \otimes \; \id}
\ar[dd]_{\id \; \otimes \; \gamma^\cC}
&&&\\
&&&&
[N,Q]\otimes[M,N] 
\ar[dd]^{\gamma^\cC}
&\\
&&
[P,Q]\otimes[M,P]
\ar[rrd]_{\gamma^\cC}
&&&\\
&&&&
[M,Q]
&
\\
&
[P,Q]\otimes[M,P]_A
\ar[ruu]_{\id \; \otimes \; \textnormal{eq}}
&&&&
\\
[P,Q]_A\otimes[M,P]_A
\ar[rrrrr]_{\gamma^A}
\ar[ru]^{\textnormal{eq} \; \otimes \; \id}
&&&&&
[M,Q]_A
\ar[luu]_{\textnormal{eq}}
}
}
 \]
 \caption{Proof of Claim~\ref{claim1}(b): Verifying the associativity of the morphisms $\gamma^A$.}
\label{fig:C_A-assoc}
\end{figure}
Here, the outer left region and top-right-interior square commutes by level-exchange; the other three outer regions and left inner region commute by the definition of $\gamma^A$ given in part (a); and the bottom-right-interior square commutes by associativity of $\gamma^\cC$. The equality of the two outer paths can be obtained by using the fact that the bottom right diagonal arrow $\textnormal{eq}:[M,Q]_A\rightarrow[M,Q]$ is a monomorphism, hence left-cancellable.  This concludes the proof of Claim~\ref{claim1}.
\end{proof}

\begin{proof}[Proof of Claim~\ref{claim2}]
(a) By the universal property of $\textnormal{eq}: [M,M]_A\rightarrow[M,M]$ it suffices to show that $\kappa^\cC_M:\one\rightarrow[M,M]$ equalizes $R_{M,M}$ and $S_{M,M}$. We will instead show that $\kappa^\cC_M\otimes \id_M \otimes \id_A$ equalizes $R^\#_{M,M}$ and $S^\#_{M,M}$ by Lemma~\ref{lem:equalizers}. This is done in Figure~\ref{fig:claim2-a}.
\begin{figure}[h!]
 \[
{\footnotesize
\begin{tikzcd}[row sep=15]
	&& {\one\otimes M\otimes A} \\
	{[M,M]\otimes M\otimes A} &&&& {[M,M]\otimes M\otimes A} \\
	&& {\one \otimes M} \\
	{[M,M]\otimes M} &&&& {M\otimes A} \\
	&& M
	\arrow["{\id \; \otimes\; \rho^M}"{pos=0.6}, from=1-3, to=3-3]
	\arrow["{\kappa^\cC_M \; \otimes \; \id \; \otimes \; \id}"'{pos=0.3}, from=1-3, to=2-1]
	\arrow["{\kappa^\cC_M \; \otimes \; \id}"'{pos=0.4}, from=3-3, to=4-1]
	\arrow["{\ell_M}", from=3-3, to=5-3]
	\arrow["{\id \; \otimes \; \rho^M}"', from=2-1, to=4-1]
	\arrow["{\epsilon^M_M}"', from=4-1, to=5-3]
	\arrow["{\kappa^\cC_M \; \otimes \; \id \; \otimes \; \id}", from=1-3, to=2-5]
	\arrow["{\epsilon^M_M \; \otimes \; \id_A}", from=2-5, to=4-5]
	\arrow["{\rho^M}", from=4-5, to=5-3]
	\arrow["{\ell_M \; \otimes \; \id_A}"{pos=0.6}, from=1-3, to=4-5]
	\arrow["{\ell_{M\otimes A}}"'{pos=0.6}, from=1-3, to=4-5]
\end{tikzcd}
}
 \]
\caption{Proof of Claim~\ref{claim1}(b): Defining the identity morphisms $\kappa^A$.}
\label{fig:claim2-a}
\end{figure}
Here, the left square commutes by level-exchange; the bottom left and top right triangles commute since $(\kappa_M^\cC)^\# = ((\ell_M)^\flat)^\# = \ell_M$ [\eqref{eq:sharp-flat}, \eqref{eq:right-closed}, Proposition~\ref{prop:self-enriched}(iii)]; and the right square commutes by naturality of the left unitor.

\skp

(b) The right identity axiom holds by Figure~\ref{fig:claim2-b-right}, and the left identity axiom holds likewise.
\begin{figure}[h!]
 \[
{\footnotesize
\xymatrix@C=4pc@R=1.2pc{
[M,N]_A\otimes\one
\ar[rrr]^{\id \; \otimes \; \kappa^A_M}
\ar[dddd]_{r_{[M,N]_A}}
\ar[rdd]_{\textnormal{eq} \; \otimes \; \id}
\ar[rrd]_{\id \; \otimes \; \kappa^\cC_M}
&&& [M,N]_A\otimes[M,M]_A
\ar[ld]^{\id \; \otimes \; \textnormal{eq}}
\ar@/^4pc/[llldddd]^{\gamma^A_{M,M,N}}
\\
&&[M,N]_A\otimes[M,M]
\ar[d]^{\textnormal{eq} \; \otimes \; \id}
&
\\
& [M,N]\otimes\one
\ar[d]_{r_{[M,N]}}
\ar[r]^(.45){\id \; \otimes \; \kappa^\cC_M}
&
[M,N] \otimes[M,M]
\ar[ld]^{\gamma^\cC_{M,M,N}}
&
\\
& [M,N] &&
\\
[M,N]_A
\ar[ur]_{\textnormal{eq}}
&&& 
}
}
 \]
\caption{Proof of Claim~\ref{claim1}(b): Verifying the right unitality of the morphisms $\kappa^A$.}
\label{fig:claim2-b-right}
\vspace{.5in}
\end{figure}
Here, the left square commutes by naturality of the right unitor; the top triangle and bottom right region commute by part (a) and Claim~\ref{claim1}(a); the center quadrant commutes by level-exchange; and the center triangle commutes by the right unit axiom of $\cC$ as a $\cC$-category. Now the two outer paths are equal  as  $\textnormal{eq}:[M,N]_A\rightarrow[M,N]$ is left-cancellable. Thus, Claim~\ref{claim2} holds.
\end{proof}

\pagebreak 

\subsection{Proof of Claims~\ref{claim3} and~\ref{claim4}
} \label{app:Sec3.4claims}

\begin{proof}[Proof of Claim~\ref{claim3}] By Lemma~\ref{lem:equalizers}, we can instead show that $m^\flat\otimes\id_{A\otimes A}$ equalizes the transposes $R^\sharp_{A,A}=\epsilon_A^A \; (\id_{[A,A]}\otimes m)$ and $S_{A,A}^\sharp=m   (\epsilon_A^A\otimes\id_A)$. This is shown below.
\[
{\small
\begin{tikzcd}[row sep=15]
	{A\otimes A\otimes A} &&& {[A,A\otimes A]\otimes A\otimes A} && {[A,A]\otimes A\otimes A} \\
	& {A\otimes A} && {[A,A\otimes A]\otimes A} && {[A,A]\otimes A} \\
	{[A,A\otimes A]\otimes A\otimes A} \\
	& {A\otimes A\otimes A} && {A\otimes A} \\
	{[A,A]\otimes A\otimes A} & {A\otimes A} &&&& A
	\arrow["{\eta_A^A\otimes\id\otimes\id}"', from=1-1, to=3-1]
	\arrow["\id"{pos=0.7}, from=1-1, to=4-2]
	\arrow["{\id\otimes m}"{pos=0.7}, from=1-1, to=2-2]
	\arrow["{\eta_A^A\otimes\id}", from=2-2, to=2-4]
	\arrow["{\epsilon_{A\otimes A}^A}", from=2-4, to=4-4]
	\arrow["\id"', from=2-2, to=4-4]
	\arrow["{\id\otimes m}"', from=4-2, to=4-4]
	\arrow["{\eta_A^A\otimes\id\otimes\id}", from=1-1, to=1-4]
	\arrow["{\id\otimes m}", from=1-4, to=2-4]
	\arrow["{[A,m]\otimes\id\otimes\id}", from=1-4, to=1-6]
	\arrow["{\id\otimes m}", from=1-6, to=2-6]
	\arrow["{[A,m]\otimes\id}", from=2-4, to=2-6]
	\arrow["{m\otimes \id}"', from=4-2, to=5-2]
	\arrow["m"', from=5-2, to=5-6]
	\arrow["m"', from=4-4, to=5-6]
	\arrow["{\epsilon_{A}^A}", from=2-6, to=5-6]
	\arrow["{[A,m]\otimes\id\otimes\id}"', from=3-1, to=5-1]
	\arrow["{\epsilon_A^A\otimes\id}"', from=5-1, to=5-2]
	\arrow["{\epsilon_{A\otimes A}^A\otimes\id}"', from=3-1, to=4-2]
\end{tikzcd}
}
\]
Here, the top two squares commute by level-exchange, the two triangles commute by adjunction triangle identities, the right-most and bottom left squares commute by naturality of $\epsilon$, and the bottom square commutes by associativity of multiplication.
\end{proof}

\begin{proof}[Proof of Claim~\ref{claim4}] 

First, we show that $\varphi$ is unital, i.e., that it satisfies $\varphi \hspace{0.02in} u = \kappa_A^A$. Notice that:
\[
\begin{array}{rll}
\varphi \; u= \kappa_A^A 
&\Leftrightarrow\hspace{0.1in} \text{eq} \; \varphi \; u = \text{eq} \; \kappa_A^A 
&\Leftrightarrow\hspace{0.1in} m^\flat \; u = \eta_\one^A \\[.4pc]
&\Leftrightarrow\hspace{0.1in} (m^\flat \; u)^\sharp = (\eta_\one^A)^\sharp 
&\Leftrightarrow\hspace{0.1in} \epsilon_A^A \;(m^\flat\otimes\id_A)\;(u\otimes\id_A) = \epsilon_A^A\;(\eta_\one^A\otimes\id_A).
\end{array}
\]
The first biconditional follows since $\text{eq}$ is a monomorphism, the second by applying definitions, the third by taking transposes, and the last by definition of $(-)^\sharp$. This last equation is proven in the diagram below.
\[{\small
\begin{tikzcd}[row sep=small]
	&&& {A\otimes A} \\
	\\
	A && {A\otimes A} && {[A,A\otimes A]\otimes A} && {[A,A]\otimes A} \\
	\\
	{[A,A]\otimes A} && A
	\arrow["{u\otimes\id}", from=3-1, to=1-4]
	\arrow["{m^\flat\otimes\id}", from=1-4, to=3-7]
	\arrow["{u\otimes\id}"{pos=0.7}, from=3-1, to=3-3]
	\arrow["\id", from=1-4, to=3-3]
	\arrow["{\eta_A^A\otimes\id}"', from=1-4, to=3-5]
	\arrow["{[A,m]\otimes\id}"', from=3-5, to=3-7]
	\arrow["{\epsilon_{A\otimes A}^A}", from=3-5, to=3-3]
	\arrow["{\eta_\one^A\otimes\id}"', from=3-1, to=5-1]
	\arrow["{\epsilon_A^A}"', from=5-1, to=5-3]
	\arrow["{\epsilon_A^A}", curve={height=-18pt}, from=3-7, to=5-3]
	\arrow["m", from=3-3, to=5-3]
	\arrow["{\id}"', from=3-1, to=5-3]
\end{tikzcd}
}\]
Here, the middle and bottom left triangles commute by adjunction triangle axioms, the right triangle commutes by definition of $m^\flat$, the right quadilateral commutes by naturality of $\epsilon$, and the interior triangle commutes by unitality of $A$. 

It remains to show $\varphi$ is multiplicative. Note that 
\[
{\small
\begin{array}{rl}
\gamma^A_{A,A,A}\;(\varphi\otimes\varphi)=\varphi\; m
&\Leftrightarrow \hspace{0.1in}
\textnormal{eq} \; \gamma^A_{A,A,A}\;(\varphi\otimes\varphi)= \textnormal{eq} \;\varphi \; m\\[.4pc]
&\Leftrightarrow \hspace{0.1in}
\gamma_{A,A,A}\;(\textnormal{eq} \otimes \textnormal{eq} )\; (\varphi\otimes\varphi)= \textnormal{eq} \;\varphi\; m\\[.4pc]
&\Leftrightarrow \hspace{0.1in}
\gamma_{A,A,A}\;(m^\flat \otimes m^\flat) = m^\flat \; m\\[.4pc]
&\Leftrightarrow \hspace{0.1in}
\left(\gamma_{A,A,A}\;([A,m]\; \eta^A_A \otimes [A,m]\; \eta^A_A)\right)^\# = \left([A,m]\; \eta^A_A \; m\right)^\#\\[.4pc]
&\Leftrightarrow \hspace{0.1in}
\epsilon_A^A \; \left(\gamma_{A,A,A} \; ([A,m]\; \eta^A_A \otimes [A,m]\; \eta^A_A)  \otimes \id_A \right) = \epsilon_A^A \;([A,m]\; \eta^A_A \; m  \otimes \id_A).
\end{array}
}
\]
%i.e.,  $\gamma^A_{A,A,A}\circ(\varphi\otimes\varphi)=\varphi\circ m$.
%Applying the monomorphism $\textnormal{eq}$ on the left, taking transposes, and unpacking definitions, this amounts to proving that:
%\[\epsilon_A^A(\gamma_{A,A,A}\otimes\id_A)([A,m]\otimes[A,m]\otimes\id_A)(\eta_A^A\otimes\eta_A^A\otimes\id_A)=\epsilon_A^A([A,m]\otimes\id_A)(\eta_A^A\otimes\id_A)(m\otimes\id_A).\]
The diagram below shows that the term $\epsilon_A^A\;(\gamma_{A,A,A}\otimes\id_A)$ %on the left hand side of the equation above may 
can be replaced with $\epsilon_A^A\;(\id_{[A,A]}\otimes\epsilon_A^A)$.
\[
\hspace{-.05in}{\footnotesize
\begin{tikzcd}[column sep=24, row sep=15]
	{[A,A]\otimes[A,A]\otimes A} && {[A,[A,A]\otimes[A,A]\otimes A]\otimes A} && {[A,[A,A]\otimes A]\otimes A} && {[A,A]\otimes A} \\
	\\
	&& {[A,A]\otimes[A,A]\otimes A} && {[A,A]\otimes A} && A
	\arrow["\eta_{[A,A]\otimes[A,A]}^A\otimes\id"', from=1-1, to=1-3]
	\arrow["{[A,\id\otimes\epsilon_A^A]\otimes\id}"', from=1-3, to=1-5]
	\arrow["{[A,\epsilon_A^A]\otimes\id}"', from=1-5, to=1-7]
	\arrow["\epsilon_A^A", from=1-7, to=3-7]
	\arrow["\id\otimes\epsilon_A^A", from=3-3, to=3-5]
	\arrow["\epsilon_A^A", from=3-5, to=3-7]
	\arrow["\epsilon_{[A,A]\otimes A}^A", from=1-5, to=3-5]
	\arrow["\epsilon_{[A,A] \otimes [A,A] \otimes A}^A", from=1-3, to=3-3]
	\arrow["\id"', from=1-1, to=3-3]
	\arrow["{\gamma_{A,A,A}}", curve={height=-30pt}, from=1-1, to=1-7]
\end{tikzcd}
}
\]
Here, the top region is the definition of $\gamma_{A,A,A}$, the left triangle commutes by an adjunction triangle identity, and the remaining squares commute by naturality of $\epsilon$.
%where we show both sides are equal to $(\id_A\otimes m)\circ m=(m\otimes\id_A)\circ m$.
So, it suffices to show that:
\[
\epsilon_A^A\;(\id_{[A,A]}\otimes\epsilon_A^A)\; ([A,m]\; \eta^A_A \otimes [A,m]\; \eta^A_A  \otimes \id_A)\; = \;\epsilon_A^A \;([A,m]\; \eta^A_A \; m  \otimes \id_A).
\]
%\[\epsilon_A^A\circ(\id_{[A,A]}\otimes\epsilon_A^A)([A,m]\otimes[A,m]\otimes\id_A)(\eta_A^A\otimes\eta_A^A\otimes\id_A)=\epsilon_A^A([A,m]\otimes\id_A)(\eta_A^A\otimes\id_A)(m\otimes\id_A).\]
We prove this by showing both sides of the equation above are equal to $m (\id_A\otimes m)$.
For the left hand side, the following diagram commutes.
\[
{\footnotesize
\begin{tikzcd}[column sep=24, row sep=15]
	&& {A\otimes A\otimes A} && {[A,A\otimes A]\otimes[A,A\otimes A]\otimes A} && {[A,A]\otimes[A,A]\otimes A} \\
	\\
	&& {[A,A\otimes A]\otimes A\otimes A} && {[A,A]\otimes A \otimes A} && {[A,A]\otimes A} \\
	{A\otimes A} && {[A,A\otimes A]\otimes A} \\
	&& {A\otimes A} &&&& A
	\arrow["\eta_A^A\otimes\eta_A^A\otimes\id", from=1-3, to=1-5]
	\arrow["{[A,m]\otimes[A,m]\otimes\id}", from=1-5, to=1-7]
	\arrow["\eta_A^A\otimes\id\otimes\id", from=1-3, to=3-3]
	\arrow["{[A,m]\otimes\id\otimes\id}", from=3-3, to=3-5]
	\arrow["{[A,m]\otimes\epsilon^A_{A\otimes A}}"', from=1-5, to=3-5]
	\arrow["{\id\otimes m}", from=3-3, to=4-3]
	\arrow["\eta_A^A\otimes\id", from=4-1, to=4-3]
	\arrow["{\id\otimes m}"', from=1-3, to=4-1]
	\arrow["\id"', from=4-1, to=5-3]
	\arrow["\epsilon_{A\otimes A}^A", from=4-3, to=5-3]
	\arrow["m", from=5-3, to=5-7]
	\arrow["{\id\otimes m}", from=3-5, to=3-7]
	\arrow["\id\otimes\epsilon_A^A", from=1-7, to=3-7]
	\arrow["{[A,m]\otimes\id}"', curve={height=12pt}, from=4-3, to=3-7]
	\arrow["\epsilon_A^A", from=3-7, to=5-7]
\end{tikzcd}
}
\]
Here, the leftmost and middle quadrangles commute by level exchange, the center top rectangle and bottom left triangle commute by an adjunction triangle identity, and the remaining regions commute by naturality of $\epsilon$.
For the right hand side, the following diagram commutes.
\[
{\footnotesize
\begin{tikzcd}[column sep=40, row sep=15]
	&& {A\otimes A\otimes A} && {A\otimes A\otimes A} \\
	\\
	{A\otimes A} && {[A,A\otimes A\otimes A]\otimes A} && {[A,A\otimes A]\otimes A} && {A\otimes A} \\
	\\
	&& {[A,A\otimes A]\otimes A} && {[A,A]\otimes A} && A
	\arrow["{\eta_{A\otimes A}^A\otimes \id}", from=1-3, to=3-3]
	\arrow["{[A,m\otimes \id]\otimes \id}", from=3-3, to=5-3]
	\arrow["{m\otimes \id}"', from=1-3, to=3-1]
	\arrow["{\eta_A^A\otimes \id}"', from=3-1, to=5-3]
	\arrow["\id", from=1-3, to=1-5]
	\arrow["\epsilon_{A\otimes A \otimes A}^A", from=3-3, to=1-5]
	\arrow["{\id\otimes m}", from=1-5, to=3-7]
	\arrow["{[A,\id\otimes m]\otimes \id}", from=3-3, to=3-5]
	\arrow["\epsilon_{A\otimes A}^A", from=3-5, to=3-7]
	\arrow["{[A,m]\otimes \id}", from=3-5, to=5-5]
	\arrow["{[A,m]\otimes \id}", from=5-3, to=5-5]
	\arrow["m", from=3-7, to=5-7]
	\arrow["\epsilon_A^A", from=5-5, to=5-7]
\end{tikzcd}
}
\]
Here, the center top triangle commutes by an adjunction triangle identity, the left region commutes by naturality of $\eta$, the top and bottom right regions commute by naturality of $\epsilon$, and the bottom center region commutes by associativity of the product $m$.

This concludes the proof of the claim.
\end{proof}

\subsection{Proof of Claim~\ref{claim5}
} \label{app:Sec5.2claims}

\begin{proof}[Proof of Claim~\ref{claim5}]
 Note that we will suppress notation in the arguments below. Namely, $\otimes:= \otimes^{\cV}$, $\cD:= \cD^{\cV}$, $\cE:= \cE^{\cV}$, and we will also suppress parenthesis and super/subscripts in many places. To proceed, let us consider three commutative diagrams attached to the morphisms $f_*^W$, $g_W^*$, for $f: Z \to Z'$ in $\cE_0$ and $g: Z' \to Z$ in $\cD_0$. First, we have that these morphisms commute with composition as illustrated in the commutative diagrams below.
\begin{equation} \label{eq:claim5-1}
{\footnotesize
\xymatrix@R=1pc@C=1.5pc{
\cE(W,Z) \otimes \cE(V,W)
\ar[rr]^{f_*^W \otimes \id}
\ar[ddd]_{\gamma}
\ar[rd]^{f \otimes \id \otimes \id}
& 
& \cE(W,Z') \otimes \cE(V,W)
\ar[ddd]^{\gamma}\\
& \cE(Z,Z') \otimes \cE(W,Z) \otimes \cE(V,W)
\ar[ru]^{\gamma \otimes \id}
\ar[d]^{\id \otimes \gamma}
&\\
& \cE(Z,Z') \otimes  \cE(V,Z)
\ar[rd]^{\gamma}\\
\cE(V,Z) 
\ar[rr]^{f_*^V}
\ar[ru]^{f \otimes \id}
& 
&
\cE(V,Z')
}
}
\end{equation}

\begin{equation} \label{eq:claim5-2}
{\footnotesize
\xymatrix@R=1pc@C=1.5pc{
\cD(W,V) \otimes \cD(Z,W)
\ar[rr]^{\id \otimes g^*_W}
\ar[ddd]_{\gamma}
\ar[rd]^{\id \otimes \id \otimes g}
& 
& \cD(W,V) \otimes \cD(Z',W)
\ar[ddd]^{\gamma}\\
& \cD(W,V) \otimes \cD(Z,W) \otimes \cD(Z',Z)
\ar[ru]^{\id \otimes \gamma}
\ar[d]^{\gamma \otimes \id}
&\\
& \cD(Z,V) \otimes  \cD(Z',Z)
\ar[rd]^{\gamma}\\
\cD(Z,V) 
\ar[rr]^{g^*_V}
\ar[ru]^{\id \otimes g}
& 
&
\cD(Z',V)
}
}
\end{equation}

\smallskip

The top and bottom regions of \eqref{eq:claim5-1} and \eqref{eq:claim5-2} commute by definition, the left regions commute by level exchange, and the right regions commute by the associativity of $\gamma$.

Now assume that the morphism $g: Z' \to Z$ is invertible in $\cD_0$. We then get the commutative diagrams below,  the first set holding by unitality.
\begin{equation} \label{eq:claim5-unit}
{\footnotesize
\xymatrix@R=2pc@C=3pc{
\one \ar[r]^(.4){g} \ar[d]_(.45){\kappa_Z} & \cD(Z',Z) \ar[d]^{\id \otimes g^{-1}} && \one \ar[r]^(.4){g} \ar[d]_(.45){\kappa_{Z'}} & \cD(Z,Z')  \ar[d]^{g^{-1} \otimes \id}\\
\cD(Z,Z)  & \cD(Z',Z)\otimes \cD(Z,Z') \ar[l]^(.58){\gamma} && \cD(Z',Z') & \cD(Z,Z')\otimes \cD(Z',Z) \ar[l]^(.58){\gamma}\\
}
}
\end{equation}
\begin{equation} \label{eq:claim5-3}
{\footnotesize
\xymatrix@R=1.5pc@C=.8pc{
\cD(Z,W) \otimes \cD(V,Z)
\ar[r]^{g^* \otimes \id}
\ar[rd]^{\id \otimes g \otimes \id}
\ar@/_2pc/[rddd]^(.35){\id \otimes \kappa_Z \otimes \id}
\ar@/_5.5pc/[ddddrr]^(.6){\gamma}
&
\cD(Z',W) \otimes \cD(V,Z)
\ar[r]^{ \id \otimes (g^{-1})_*}
\ar[rd]^{\id \otimes g^{-1} \otimes \id}
&
\cD(Z',W) \otimes \cD(V,Z')
\ar@/^8pc/[dddd]_{\gamma}\\
&
\cD(Z,W) \otimes \cD(Z',Z) \otimes \cD(V,Z)
\ar[u]^{\gamma \otimes \id}
\ar[d]^(.45){\id \otimes \id \otimes g^{-1} \otimes \id}
&
\cD(Z',W) \otimes \cD(Z,Z') \otimes \cD(V,Z)
\ar[u]^{\id \otimes \gamma}
\ar[dd]^{\gamma \otimes \id}\\
&
\cD(Z,W) \otimes \cD(Z',Z) \otimes \cD(Z,Z') \otimes \cD(V,Z) \hspace{-.5in}
\ar[ru]_(.6){\gamma \otimes \id \otimes \id}
\ar[d]^{\id \otimes \gamma \otimes \id}
&
\\
&
\cD(Z,W) \otimes \cD(Z,Z) \otimes \cD(V,Z)
\ar@/_1pc/[r]_(.5){\gamma \otimes \id} \hspace{-.5in}
& \cD(Z,W) \otimes  \cD(V,Z) 
\ar[d]^{\gamma}\\
&
&
\cD(V,W)
}
}
\end{equation}

\medskip

For \eqref{eq:claim5-3}, the top triangles commute by definition, the left-middle region commutes by \eqref{eq:claim5-unit}, and the rest of the regions commute by the associativity of $\gamma$ or by level exchange.

Now the proof of the claim follows from the commutative diagrams in Figures~\ref{fig:claim5} and~\ref{fig:claim5-2}. In  Figure~\ref{fig:claim5}, the top and bottom regions commute by definition, and the remaining unlabelled regions commute because $\cV$-functors are compatible with $\gamma$. In Figure~\ref{fig:claim5-2}, the left region commutes by level exchange, the second top region commutes by the unitality of $\gamma$, the third top region commutes by \eqref{eq:claim5-unit},  the right region commutes because $\cV$-functors are compatible with $\kappa$, and the bottom regions commute by definition.
\end{proof}

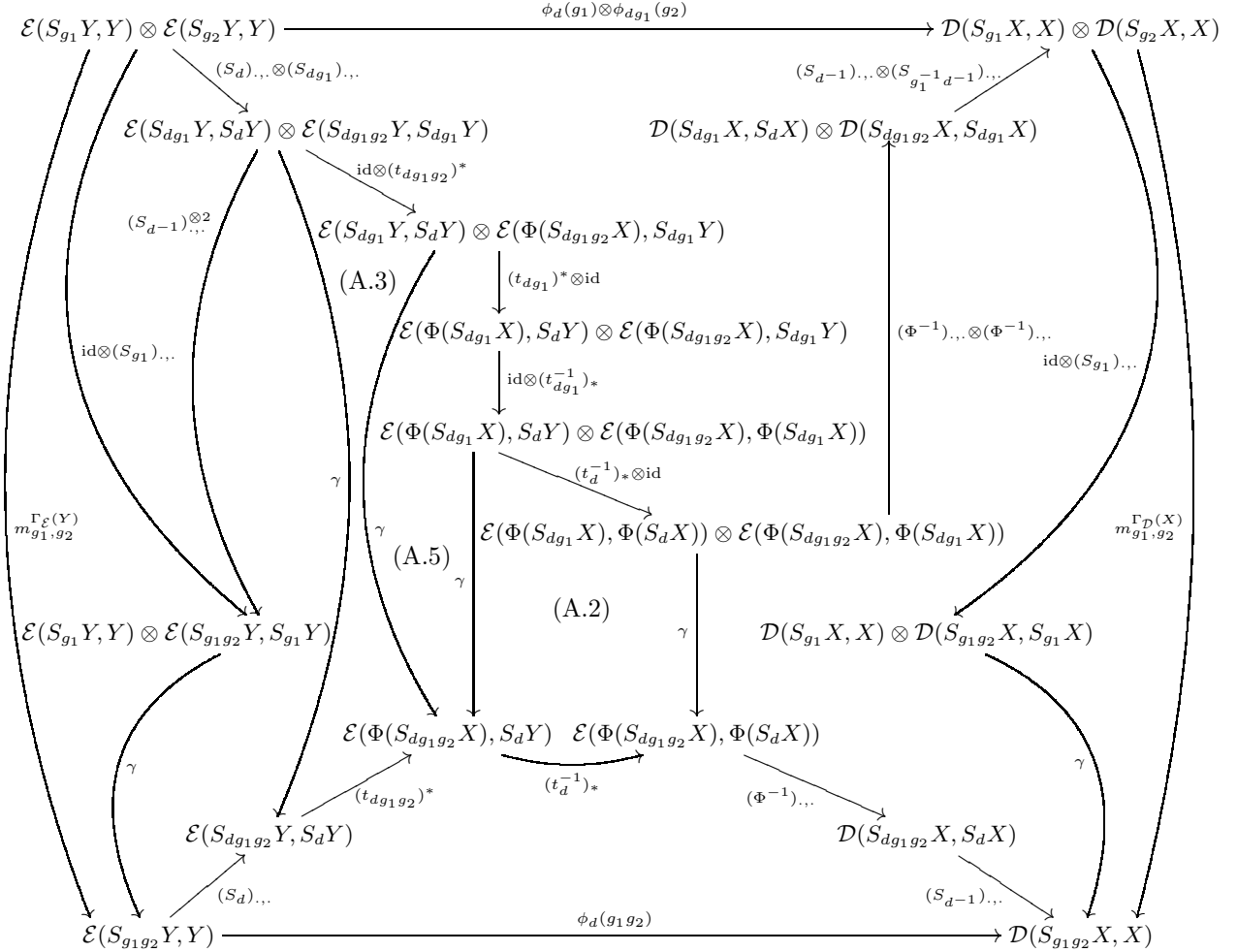
\begin{figure}
 \[
{\footnotesize
\xymatrix@R=2pc@C=-5.5pc{
\cE(S_{g_1}Y,Y) \otimes \cE(S_{g_2}Y,Y)
\ar[rrrrr]^{\phi_d(g_1) \otimes \phi_{d g_1}(g_2)}
\ar@/_3pc/@<-20pt>[ddddddddd]^(.55){m_{g_1, g_2}^{\Gamma_\cE(Y)}}
\ar[rd]^(.6){(S_d)_{.,.} \otimes (S_{d g_1})_{.,.}}
\ar@/_4.5pc/[rdddddd]^{\id \otimes (S_{g_1})_{.,.}}
&
&
&
&
&
 \cD(S_{g_1}X,X) \otimes \cD(S_{g_2}X,X)
 \ar@/^2pc/@<20pt>[ddddddddd]_(.55){m_{g_1, g_2}^{\Gamma_\cD(X)}}
 \ar@/^4.8pc/[ldddddd]_{\id \otimes (S_{g_1})_{.,.}}\\
&
\cE(S_{d g_1}Y,S_d Y) \otimes \cE(S_{d g_1 g_2}Y,S_{d g_1} Y) \hspace{-.4in}
\ar[rd]^(.6){\id \otimes (t_{d g_1 g_2})^*}
\ar@/_2.5pc/[ddddd]_(.2){(S_{d^{-1}})_{.,.}^{\otimes 2}}
\ar@/^2.6pc/[ddddddd]_{\gamma}
&
&
&
\hspace{-.9in} \cD(S_{d g_1}X,S_d X) \otimes \cD(S_{d g_1 g_2}X, S_{d g_1} X)
\ar[ru]^(.3){(S_{d^{-1}})_{.,.} \otimes (S_{g_1^{-1} d^{-1}})_{.,.}}
&\\
&
&
 \hspace{.3in} \cE(S_{d g_1}Y,S_d Y) \otimes \cE(\Phi(S_{d g_1 g_2}X), S_{d g_1} Y)  \hspace{-.5in}
 \ar@<20pt>[d]^{(t_{d g_1})^* \otimes \id}
 \ar@/_2.75pc/[ddddd]^(.6){\gamma}
&
&
&\\
&
&
\cE(\Phi(S_{d g_1}X),S_d Y) \otimes \cE(\Phi(S_{d g_1 g_2}X), S_{d g_1} Y) \hspace{-1.9in}
\ar@<20pt>[d]^{\id \otimes (t_{d g_1}^{-1})_*} 
&
&
&\\
&
&
 \cE(\Phi(S_{d g_1}X),S_d Y) \otimes \cE(\Phi(S_{d g_1 g_2}X), \Phi(S_{d g_1} X)) \hspace{-1.9in}
 \ar[rd]^(.6){(t_{d}^{-1})_* \otimes \id}
 \ar@<10pt>[ddd]_{\gamma}
&
&
&\\
&
&
&
 \cE(\Phi(S_{d g_1}X),\Phi(S_d X)) \otimes \cE(\Phi(S_{d g_1 g_2}X), \Phi(S_{d g_1} X)) \hspace{-.5in}
 \ar@<-75pt>[uuuu]_(.5){(\Phi^{-1})_{.,.} \otimes (\Phi^{-1})_{.,.}}
   \ar[dd]_{\gamma}
&
&\\
&
\hspace{-1in} \cE(S_{g_1}Y,Y) \otimes \cE(S_{g_1 g_2}Y,S_{g_1} Y) 
\ar@/_3pc/[lddd]^(.55){\gamma}
&
&
&
\cD(S_{g_1}X,X) \otimes \cD(S_{g_1 g_2}X, S_{g_1} X)
\ar@/^3pc/[rddd]_(.55){\gamma}
&\\
&
&
\cE(\Phi(S_{d g_1 g_2}X), S_d Y) 
\ar@/_1pc/[r]_{(t_d^{-1})_*}
&
\cE(\Phi(S_{d g_1 g_2}X),\Phi(S_d X))
\ar[rd]_(.45){(\Phi^{-1})_{.,.}}
&
&
\\
&
\cE(S_{d g_1 g_2}Y, S_d Y)
\ar[ru]_(.55){(t_{d g_1 g_2})^*}
&
&
&
\cD(S_{d g_1 g_2}X,S_d X)
\ar[rd]_(.45){(S_{d^{-1}})_{.,.}}
&
\\
\cE(S_{g_1 g_2}Y,Y) 
\ar[rrrrr]^{\phi_d(g_1 g_2)}
\ar[ru]_(.55){(S_d)_{.,.}}
\ar@{{ }{ }}@/_1.7pc/[uuuuuuuuurr]^(.73){\hbox{\small{\eqref{eq:claim5-2}}}}
\ar@{{ }{ }}[uuuuuuuuurrr]_(.43){\hbox{\small{\eqref{eq:claim5-3}}}}
\ar@{{ }{ }}[uuuuuuuurrrrr]_(.43){\hbox{\small{\eqref{eq:claim5-1}}}}
&
&
&
&
&
\cD(S_{g_1 g_2}X,X)
}
}
\]

\caption{Proof of Claim~\ref{claim5}: Verifying the twist condition \eqref{eq:phi-app}.}
\label{fig:claim5}
\end{figure}

\vspace{.2in}

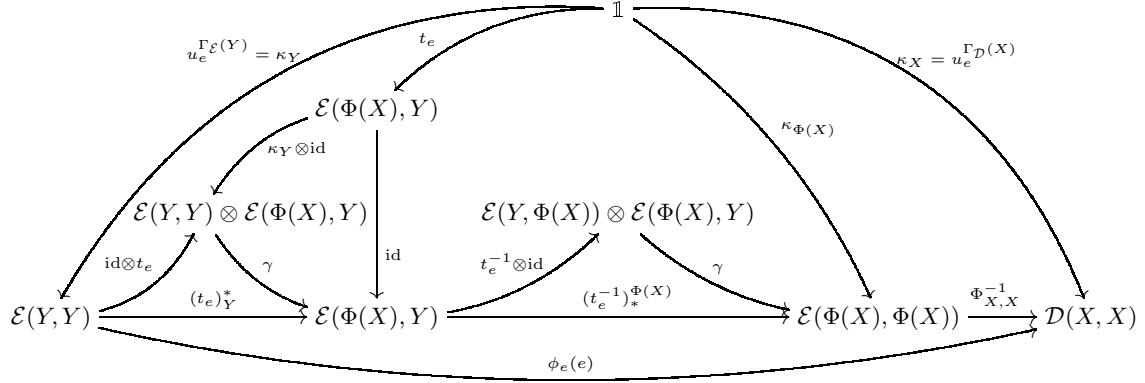
\begin{figure}
 \[
{\footnotesize
\xymatrix@R=2pc@C=0.8pc{
&&& \one \ar@/_3.3pc/[lllddd]_{u_e^{\Gamma_\cE(Y)} \; = \;\kappa_Y \; } \ar@/^3.3pc/[rrrddd]^{\kappa_X \; = \; u_e^{\Gamma_\cD(X)}} \ar@/_1pc/[ld]_(.7){t_e} \ar@/^1pc/[rddd]^{\kappa_{\Phi(X)}}  &&& \\
&& \cE(\Phi(X),Y) \ar@/_1pc/[ld]^(.5){\kappa_Y \otimes \id} \ar[dd]^(.7){\id}   &&&& \\
& \cE(Y,Y) \otimes \cE(\Phi(X),Y) \hspace{-.5in}
\ar@/_1pc/[rd]^(.4){\gamma} && \cE(Y,\Phi(X)) \otimes \cE(\Phi(X),Y) \ar@/_1pc/[rd]^(.4){\gamma} &&&\\
\cE(Y,Y) \ar@/_1pc/[ru]^(.6){\id \otimes t_e} \ar[rr]^{(t_e)^*_Y} \ar@/_2pc/[rrrrrr]^(.5){\phi_e(e)}&& \cE(\Phi(X),Y)  \ar@/_1pc/[ru]^(.6){t_e^{-1} \otimes \id} \ar[rr]^{(t_e^{-1})_*^{\Phi(X)}}  && \cE(\Phi(X),\Phi(X)) \ar[rr]^(.55){\Phi_{X,X}^{-1}} && \cD(X,X)
}
}
\]

\caption{Proof of Claim~\ref{claim5}: Verifying the twist condition \eqref{eq:phi-app-2}.}
\label{fig:claim5-2}
\end{figure}

\section*{Acknowledgements}

The authors thank César Galindo for insightful discussions at the beginning of this project.
Walton was partially
supported by the US NSF grants
\#DMS-2100756 and \#DMS-2348833, and the Alexander von Humboldt Foundation. She was hosted by the University of Hamburg during many phases of this project, and would like to thank her hosts for providing excellent working conditions.  We also thank the referee for their careful comments, which improved  our manuscript.

%%%%%%%%%%%%%%%%%%%%%%%%%%
%%%%%%%%%%%%%%%%%%%%%%%%%%
%%%%%%%%%%%%%%%%%%%%%%%%%%

\bibliography{TwistMonoidalCats}
\bibliographystyle{alpha}

\end{document}